\documentclass[11pt,oneside,letterpaper]{amsart}

\usepackage[margin=1in]{geometry}
\usepackage{amsmath,amssymb,amsthm}
\usepackage{mathrsfs}
\usepackage{xcolor}
\usepackage{hyperref}
\usepackage{enumitem}

\hypersetup{
  colorlinks=true,
  linkcolor=blue,
  citecolor=red,
  urlcolor=blue
}

\newtheorem{theorem}{Theorem}[section]
\newtheorem{lemma}[theorem]{Lemma}
\newtheorem{corollary}[theorem]{Corollary}
\newtheorem{proposition}[theorem]{Proposition}
\newtheoremstyle{uprightremark}%
  {3pt}{3pt}
  {\normalfont}{}
  {\bfseries}{.}{.5em}{}
\theoremstyle{uprightremark}
\newtheorem{remark}[theorem]{Remark}
\theoremstyle{plain}

\newcommand{\cG}{\mathcal{G}}
\newcommand{\cH}{\mathcal{H}}
\newcommand{\cF}{\mathcal{F}}
\newcommand{\mult}{\operatorname{mult}}
\newcommand{\vb}[1]{\mathbf{#1}}

\begin{document}
\raggedbottom
\emergencystretch=4em

\title[Connected Fair Detachments II]{Connected Fair Detachments of Hypergraphs II}
\author{Amin Bahmanian}
\address{Department of Mathematics, Illinois State University, Normal, IL 61790-4520, USA}
\email{mbahman@ilstu.edu}

\begin{abstract}
We study embeddings of factorizations of the $\lambda$-fold complete
$h$-uniform hypergraph $\lambda K_m^h$ in factorizations of $\lambda K_n^h$,
where $m<n$ and both the source and target degrees may vary by color.
From every embedding one can obtain another that minimizes the number of
components in every target color without increasing its multiplicity spread.
This change also does not increase the sum of any convex function of
the edge multiplicities, simultaneously in all colors.

For $h=2,3$ we give exact existence criteria for every $m<n$.
For $h\ge4$, the necessary divisibility and degree-sum conditions are
sufficient once $n\ge(h-1)m$, improving the previously known threshold
$n\ge hm$ for color-dependent degrees.  In this range we also give exact criteria for connected,
simple, and equimultiple target factors.  The proofs use a system of
integer counts recording the number of added edges of each type.
Every system satisfying these equations gives an embedding with the minimum
multiplicity spread allowed by those counts and the minimum number of
components in every color.
\end{abstract}

\maketitle

\section{Introduction}

Fix $n,m,h,\lambda\in\mathbb N$ with $n>m\ge h\ge2$.
Here $\mathbb N$ denotes the positive integers, and
$[k]=\{1,\dots,k\}$ for $k\in\mathbb N$.
Write $K_v^h$ for the complete $h$-uniform hypergraph on $v$ vertices
and $\lambda K_v^h$ for the hypergraph in which every $h$-set has
multiplicity $\lambda$.

Let $\cG$ be a prescribed factorization of $\lambda K_m^h$, regarded
as an edge-colored hypergraph, with $q$ nonempty colors of positive
degrees $r_1,\dots,r_q$.  Write $G_i$ for its color-$i$ factor.
Prescribe target degrees $s_1,\dots,s_k\in\mathbb N$ with $q\le k$.
For $i\in[k]\setminus[q]$, put $r_i=0$ and let $G_i$ be the empty
spanning hypergraph on the $m$ original vertices.  Assume $r_i\le s_i$ for
$i\in[k]$.  We call the colors in $[q]$ the \emph{source colors},
and those in $[k]\setminus[q]$ the \emph{additional target colors}.

An \emph{embedding} of $\cG$ is a factorization
$\cF=\{F_1,\dots,F_k\}$ of $\lambda K_n^h$ in which $F_i$ is
$s_i$-regular and its restriction to the original $m$ vertices is
$G_i$ for $i\in[k]$.
Let $\mathscr E$ be the set of embeddings with these prescribed target degrees.
For a target factor $F_i$, write $c(F_i)$ for its number of components and
\[
 \operatorname{spr}(F_i)
 =\max\nolimits_{e\in E(K_n^h)}\mult_{F_i}(e)
  -\min\nolimits_{e\in E(K_n^h)}\mult_{F_i}(e).
\]
The extrema include underlying edges of multiplicity zero.

Every embedding must satisfy
\[
 h\mid ns_i\quad\text{for }i\in[k],
 \qquad
 \sum\nolimits_{i\in[k]}s_i=\lambda\binom{n-1}{h-1}.
\]
We call these two conditions \emph{admissibility}, always together with
the standing assumptions above.

\begin{theorem}\label{CCAII:thm:main}
For every embedding $\cF\in\mathscr E$, there is an embedding
$\cF'\in\mathscr E$ whose factors simultaneously have the minimum
possible component counts and satisfy
$\operatorname{spr}(F_i')\le\operatorname{spr}(F_i)$ for $i\in[k]$.
\end{theorem}

Admissibility is sufficient in the following range.

\begin{theorem}[Admissibility threshold]\label{CCAII:thm:main-existence}
Suppose either $h\in\{2,3\}$ and $n\ge2m$, or $h\ge4$ and
$n\ge(h-1)m$.  Then $\mathscr E\ne\varnothing$ if and only if
admissibility holds.
\end{theorem}

Theorem~\ref{CCAII:thm:main} follows from
Lemma~\ref{CCAII:lem:component-optimal}, and
Theorem~\ref{CCAII:thm:main-existence} follows from
Corollary~\ref{CCAII:cor:h23-large-order} and
Theorem~\ref{CCAII:thm:large-order-optimal}.
Whenever $\mathscr E\ne\varnothing$, every real-valued function of the
$k$ spreads that is nondecreasing in each coordinate can be minimized
while all component counts are minimum.
The replacement in Lemma~\ref{CCAII:lem:component-optimal} also does not
increase the sum of any convex function of the edge multiplicities in
any color.

Part~I~\cite{Bahmanian2026CFD} provides fair detachments controlling
degrees, multiplicities, and components.
Here they turn integer counts into embeddings, where the counts record,
for each color, how many added edges contain each possible number of new vertices.
Once these counts are fixed, a single embedding simultaneously attains the
minimum permitted spread and the minimum component count in every color.
For $h=2,3$, this gives exact existence criteria for every $m<n$ and
exact spread information when the counts are forced.
Individual minimum spreads need not be simultaneously attainable when
$h=3$, so the colors must be optimized jointly.
For $h\ge4$ and $n\ge(h-1)m$, the criterion reduces to admissibility,
and we also obtain exact conditions for connected, simple, and
equimultiple target factors.

These results continue a classical embedding problem.
Baranyai proved that $K_n^h$ is $1$-factorizable whenever
$h\mid n$~\cite{Baranyai1975}.
Cameron asked when partial parallelisms embed in
parallelisms~\cite{Cameron1976}, and Baranyai and Brouwer studied the
corresponding problem for complete uniform
hypergraphs~\cite{BaranyaiBrouwer1977}.
H\"aggkvist and Hellgren developed further embedding
methods~\cite{HH1993}.
For $\lambda=1$ and equal source and target degrees, Bahmanian and
Newman proved sufficiency for $n\ge2m$ under
$\gcd(m,n,h)=\gcd(n,h)$, giving a complete criterion in particular
when $h$ is prime~\cite{BN2018}.
Further results for the $4$-uniform equal-degree problem were obtained
in~\cite{BH2021}.
For color-dependent degrees, connected embeddings for $n\ge hm$ were
obtained in~\cite{BJN2022}; the bound $n\ge(h-1)m$ improves this
threshold for $h\ge4$.
Exact connected criteria for partial factorizations in the range
$n\ge(h-1)(2m-1)$ were obtained in~\cite{BJY2025}.
For graphs, Hilton--Johnson--Rodger--Wantland obtained an exact
connected embedding criterion for simple complete
graphs~\cite{HJRW2003}, and Johnson generalized it to prescribed
factor degrees and edge-connectivities~\cite{Johnson2007}.

Section~\ref{CCAII:sec:preliminaries} recalls fair detachment.
Section~\ref{CCAII:sec:factorization-consequences} proves the
nonuniform factorization result used in
Section~\ref{CCAII:sec:regular-embedding} to characterize embeddings
through their added-edge counts.
Sections~\ref{CCAII:sec:h23-main} and~\ref{CCAII:sec:large-order}
treat $h=2,3$ and $h\ge4$, respectively.
Section~\ref{CCAII:sec:further-consequences} determines possible
target degrees and compares iterated embeddings with direct embeddings.

\section{Preliminaries and fair detachment}\label{CCAII:sec:preliminaries}

For $x,y\in\mathbb R$, write $x\approx y$ when
$\lfloor y\rfloor\le x\le\lceil y\rceil$.
For $a,b\in\mathbb N$ and $x\in\mathbb R$, we use
$\lfloor\lfloor x/a\rfloor/b\rfloor=\lfloor x/(ab)\rfloor$ and
$\lceil\lceil x/a\rceil/b\rceil=\lceil x/(ab)\rceil$.

The edge set of a hypergraph is a multiset, and an edge may itself be a
multiset of vertices.  Thus both parallel edges and repeated occurrences
of a vertex within an edge are allowed.
The size of an edge counts these occurrences.
A hypergraph is \emph{$h$-uniform} if every edge has size $h$,
and \emph{loopless} if no edge repeats a vertex.
The vertex multiset obtained by forgetting the choice of parallel
copy is the \emph{underlying edge}.
Write $\mult_H(e)$ for its number of copies in $H$.
The degree $\deg_H(v)$ is the total number of occurrences of $v$
among the edges.

Connectivity is in the Berge sense: vertices $v_0,v_t$ are joined
if there is an alternating sequence
$v_0,e_1,v_1,\dots,e_t,v_t$ with $v_{j-1},v_j\in e_j$ for $j\in[t]$.
We count isolated vertices as components.
A hypergraph is \emph{nontrivial} if it has at least one edge.

A $k$-coloring is a map from the edge copies to $[k]$, and $H_i$
is the spanning subhypergraph formed by the edges of color $i$.
A $\vb d$-factorization, where $\vb d=(d_1,\dots,d_k)$, has
$H_i$ $d_i$-regular for $i\in[k]$.
If all degrees are $d$, we call it a $d$-factorization.
In particular, a factor of degree zero is the empty spanning
hypergraph and has one component for each vertex.

A loopless hypergraph is \emph{simple} if every underlying edge has
multiplicity at most one, and has \emph{multiplicity at most $b$}
if every multiplicity is at most $b$.
A loopless $h$-uniform hypergraph $H$ is \emph{equimultiple} if
$|\mult_H(e)-\mult_H(f)|\le1$ for all
$e,f\in\binom{V(H)}h$.
For a nonuniform hypergraph, equimultiplicity is required separately
within each edge size.  Every simple hypergraph is equimultiple
within each edge size.

For a factorization of $\lambda K_n^h$, requiring
$\mult_{F_i}(e)\approx\lambda/k$ for every color and underlying edge
makes every factor equimultiple.  When all factors have the same degree
$s$, their average edge multiplicity is
$s/\binom{n-1}{h-1}=\lambda/k$, so equimultiplicity implies the same
condition.  For unequal degrees, however, the color averages may differ,
and the two conditions need not coincide.

A \emph{detachment} replaces a vertex by several vertices and
distributes its occurrences among them, preserving the edge copies,
their sizes, and their colors.
The reverse operation is an \emph{amalgamation}.
A detachment is \emph{fair} when degrees and multiplicities are
distributed as evenly as possible in the whole hypergraph and in
each color.

A vertex $\alpha$ of a connected hypergraph $H$ is a \emph{cut vertex}
if $H=I\cup J$ for nontrivial subhypergraphs $I,J$ with
$V(I\cap J)=\{\alpha\}$ and $E(I\cap J)=\varnothing$.
An \emph{$\alpha$-wing} is a nontrivial connected subhypergraph $W$
containing $\alpha$, with $\alpha$ not a cut vertex of $W$, such that
no edge outside $W$ is incident with a vertex of
$V(W)\setminus\{\alpha\}$.
Write $\omega_\alpha(H)$ for the number of $\alpha$-wings, computed
in the component containing $\alpha$ if $H$ is disconnected.
An isolated $\alpha$ has no wings.
The wings partition the edge copies in that component, so their
degrees at $\alpha$ sum to $\deg_H(\alpha)$; see
Part~I~\cite[Section~3]{Bahmanian2026CFD}.

If $S$ is a vertex set not containing $\alpha$, write $S\alpha^r$
for the multiset obtained by adjoining $r$ copies of $\alpha$ to $S$.
In particular, $\alpha^r$ has $r$ occurrences of $\alpha$.

\begin{lemma}[{\cite[Theorem~1.1]{Bahmanian2026CFD}}]\label{CCAII:mainthconfdet}
Let $\cH$ be a $k$-colored hypergraph, $\alpha\in V(\cH)$, and
$g\in\mathbb N$.  Suppose no edge contains $\alpha$ more than $g$ times.
Then $\alpha$ can be split into a set $A$ of $g$ vertices to obtain
$\cF$ in which no edge repeats a vertex of $A$.
For $i\in[k]$, $x\in A$, nonempty $U\subseteq A$, and
$S\subseteq V(\cH)\setminus\{\alpha\}$,
\[
 \begin{aligned}
 \textup{(F1)}\quad&
 \deg_{F_i}(x)\approx\deg_{H_i}(\alpha)/g,\\
 \textup{(F2)}\quad&
 \mult_{F_i}(S\cup U)
 \approx \mult_{H_i}(S\alpha^{|U|})\big/\binom g{|U|},\\
 \textup{(C)}\quad&
 c(F_i)=c(H_i)
 \quad\Longleftrightarrow\quad
 \deg_{H_i}(\alpha)-\omega_\alpha(H_i)\ge g-1.
 \end{aligned}
\]
The same degree and multiplicity conclusions hold after omitting
the color index.
\end{lemma}

We refer to \textup{(F1)} as \emph{degree fairness},
\textup{(F2)} as \emph{multiplicity fairness}, and
\textup{(C)} as the \emph{connectivity condition}.

\section{Nonuniform regular factorizations}\label{CCAII:sec:factorization-consequences}

In this section, fix $n,p,k,\lambda_1,\dots,\lambda_p\in\mathbb N$
with $n\ge2$, and edge sizes $1\le h_1<\cdots<h_p\le n$.
For $i\in[k]$, prescribe $d_i,a_{i1},\dots,a_{ip}\in\mathbb N\cup\{0\}$,
where $a_{ij}$ is the number of edges of size $h_j$ in color $i$.

\begin{theorem}[Nonuniform factorization]\label{CCAII:thm:balanced-R}
The hypergraph $\bigcup\nolimits_{j\in[p]}\lambda_jK_n^{h_j}$
has a factorization $\{F_1,\dots,F_k\}$ in which $F_i$ is $d_i$-regular,
is equimultiple within each edge size, and has $a_{ij}$ edges of size
$h_j$, if and only if
\[
 \sum\nolimits_{j\in[p]}h_ja_{ij}=nd_i\quad\text{for }i\in[k],
 \qquad
 \sum\nolimits_{i\in[k]}a_{ij}=\lambda_j\binom n{h_j}
       \quad\text{for }j\in[p].
\]
Whenever these conditions hold, the same factorization may be chosen so that
\[
 F_i\text{ is connected}
 \quad\Longleftrightarrow\quad
 \sum\nolimits_{j\in[p]}a_{ij}\le nd_i-n+1
 \quad\text{for }i\in[k].
\]
\end{theorem}

\begin{proof}
The two equations count, respectively, incidences in each color and
edge copies of each size.
If $F_i$ is connected, its incidence bipartite graph is connected.
It has $n+\sum\nolimits_{j\in[p]}a_{ij}$ vertices and $nd_i$ edges,
so $nd_i\ge n+\sum\nolimits_{j\in[p]}a_{ij}-1$.

For sufficiency, form a colored hypergraph $\cH$ on one vertex
$\alpha$, with $a_{ij}$ copies of $\alpha^{h_j}$ in color $i$.
Every edge has at most $n$ occurrences of $\alpha$.
Also $\deg_{H_i}(\alpha)=nd_i$ and
$\omega_\alpha(H_i)=\sum\nolimits_{j\in[p]}a_{ij}$, since each edge
copy is one wing.
Apply Lemma~\ref{CCAII:mainthconfdet}, splitting $\alpha$ into $n$ vertices.
Degree fairness gives degree $d_i$ at each vertex in color $i$,
and edge-copy preservation gives the prescribed numbers $a_{ij}$.
For every $h_j$-set $U$, multiplicity fairness gives
$\mult_{F_i}(U)\approx a_{ij}/\binom n{h_j}$, so each factor is
equimultiple within that size.
After omitting the color index, it gives multiplicity
$(\sum\nolimits_{i\in[k]}a_{ij})/\binom n{h_j}=\lambda_j$ exactly.
Thus the uncolored hypergraph is the required union.

Since $c(H_i)=1$ on its one-vertex set, the connectivity condition
gives the asserted equivalence in every color simultaneously.
If $d_i=0$, the incidence equation forces $a_{ij}=0$ for every $j$,
and $F_i$ is the empty spanning hypergraph; the connectivity
inequality then fails, as required.
\end{proof}

If $h_1\ge2$, a color of degree $d_i\ge2$ can be made connected:
the incidence equation gives
$\sum\nolimits_{j\in[p]}a_{ij}\le nd_i/2\le nd_i-n+1$.
For $d_i=1$, the connectivity inequality requires one edge,
which must be spanning.  Degree-zero colors are empty.

The construction also respects bounds $b_{ij}\in\mathbb N$ on
multiplicity within size $h_j$ exactly when
$a_{ij}\le b_{ij}\binom n{h_j}$.
Necessity follows by summing the bound, and sufficiency follows
because the ceiling of $a_{ij}/\binom n{h_j}$ is at most $b_{ij}$.
In the uniform case, write $h=h_1\ge2$ and $\lambda=\lambda_1$.
An equimultiple factorization of
$\lambda K_n^h$ with degree $d_i$ and multiplicity at most
$b_i\in\mathbb N$ in color $i$ exists exactly when
\[
 h\mid nd_i,\qquad
 d_i\le b_i\binom{n-1}{h-1}\quad\text{for }i\in[k],
 \qquad
 \sum\nolimits_{i\in[k]}d_i=\lambda\binom{n-1}{h-1}.
\]
The factors may simultaneously be chosen connected exactly for the
colors with $d_i\ge2$, or $d_i=1$ and $h=n$.
Taking $b_i=1$ gives simplicity.
A factor of constant edge multiplicity requires
$\binom{n-1}{h-1}\mid d_i$, and is then
$(d_i/\binom{n-1}{h-1})K_n^h$.

For the hypergraph with one copy of every edge of sizes $1,\dots,h$,
Berge and Johnson solved the $1$-factorization problem for $h\le4$,
and Brouwer and Tijdeman gave a complete criterion for arbitrary
$h$~\cite{BergeJohnson1977,BrouwerTijdeman1981}.

\section{Components and multiplicity spread}\label{CCAII:sec:regular-embedding}

We return to the embedding notation of Section 1.
First we minimize components while preserving multiplicity bounds.
We then determine the minimum spread for each fixed choice of
added-edge counts.

The vertices of $\cG$ are \emph{old}, and the other $n-m$ vertices
are \emph{new}.
An edge is \emph{old} if all its vertices are old and \emph{added}
otherwise; it is \emph{mixed} if it contains both old and new vertices.
An added edge containing $t$ new vertices is a \emph{type-$t$ edge}.
Put $A=\binom{m-1}{h-1}$, $B=\binom{n-1}{h-1}$,
$\tau=\min\{h,n-m\}$, and
$C_t=\binom m{h-t}\binom{n-m}t$ for $t\in[\tau]$.
Thus $C_t$ counts underlying type-$t$ edges, and $C_t>0$ for $t\in[\tau]$.

\subsection{Minimum component counts}

Assuming admissibility, put
\[
 \kappa_i=
 \begin{cases}
  1, & i\in[q],\ s_i>r_i,\\
  c(G_i)+1, & i\in[q],\ s_i=r_i\ge2,\\
  c(G_i)+(n-m)/h, & i\in[q],\ s_i=r_i=1,\\
  1, & i\in[k]\setminus[q],\ s_i\ge2,\\
  n/h, & i\in[k]\setminus[q],\ s_i=1.
 \end{cases}
\]
The quotients are integers: admissibility gives $h\mid n$ when
$s_i=1$, and $G_i$ is a $1$-factor, so $h\mid m$, when $r_i=1$.

\begin{lemma}[Component minimization]\label{CCAII:lem:component-optimal}
For every embedding $\cF$ of $\cG$, there is an embedding $\cF'$
with $c(F_i')=\kappa_i$ for $i\in[k]$.
The replacement preserves the number of added edges of each type
in every color.
For specified colors, any integer multiplicity bounds $L_i\le U_i$ for
$F_i$ remain valid for $F_i'$.
For every color and every convex function $\varphi:\mathbb R\to\mathbb R$,
the same replacement satisfies
\[
 \sum\nolimits_{e\in E(K_n^h)}
       \varphi\bigl(\mult_{F_i'}(e)\bigr)
 \le
 \sum\nolimits_{e\in E(K_n^h)}
       \varphi\bigl(\mult_{F_i}(e)\bigr).
\]
In particular, $\kappa_i$ is the minimum component count,
and simplicity or equimultiplicity in specified colors may be preserved.
\end{lemma}

\begin{proof}
Put $N=n-m$.
Amalgamate the new vertices of the given embedding into $\alpha$,
keeping the old vertices and edges fixed, to obtain $\cH$.
Every edge contains at most $N$ occurrences of $\alpha$.
Apply Lemma~\ref{CCAII:mainthconfdet} and split $\alpha$ into $N$ vertices.
The resulting hypergraph $\cF'$ preserves the old edges and the number of added edges of each type in every color.
Since $\deg_{H_i}(\alpha)=Ns_i$, degree fairness gives degree $s_i$
at every new vertex.
For an old $(h-t)$-set $S$,
$\mult_{\cH}(S\alpha^t)=\lambda\binom Nt$, so multiplicity fairness
restores multiplicity $\lambda$ on every added edge.
Old-vertex degrees are unchanged.
Hence $\cF'$ is an embedding with the same target degrees.

We check components in three cases.
\begin{enumerate}[label=\textup{(\roman*)},leftmargin=1.7em,labelsep=.35em,itemsep=3pt,topsep=2pt]
\item Suppose $i\in[q]$ and $s_i>r_i$, or
$i\in[k]\setminus[q]$ and $s_i\ge2$.
Every old vertex is incident with an added edge of color $i$,
so $H_i$ is connected.

Every $\alpha$-wing has degree at least two at $\alpha$.
A wing with no old vertex contains an $\alpha^h$-edge.
For a source color, suppose a wing had only one occurrence of $\alpha$,
on an edge $e$.
Choose an old vertex $v\in e$ and an old edge $f$ containing $v$.
The wing condition forces $f$ into the wing.
Since $f$ has $h$ old vertices and $e$ has at most $h-1$,
choose $w\in f\setminus e$.
An added edge through $w$ is also forced into the wing and gives
a second occurrence of $\alpha$, a contradiction.
For an additional target color, every edge through an old vertex is added;
a wing containing that vertex therefore has degree at least $s_i\ge2$
at $\alpha$.
Since the wing degrees sum to $Ns_i$,
\[
 \deg_{H_i}(\alpha)-\omega_\alpha(H_i)
 \ge Ns_i/2\ge N\ge N-1.
\]
The connectivity condition gives $c(F_i')=1$.

\item Suppose $i\in[q]$ and $s_i=r_i$.
No added edge of color $i$ meets an old vertex.
Thus $H_i$ consists of $G_i$ and $Ns_i/h$ copies of $\alpha^h$.
The new vertices have positive degree, so existence of the given
embedding forces $N\ge h$ in this case.
If $s_i\ge2$, then
$\omega_\alpha(H_i)=Ns_i/h$ and
$\deg_{H_i}(\alpha)-\omega_\alpha(H_i)
=Ns_i(1-1/h)\ge N-1$.
The new part can therefore be made connected, giving $c(G_i)+1$
components.
If $s_i=1$, the new part is a $1$-factor and has $N/h$ components.

\item Suppose $i\in[k]\setminus[q]$ and $s_i=1$.
The target factor is a $1$-factor on all $n$ vertices,
so it has $n/h$ components.
\end{enumerate}

These counts are minimum.
One component is the absolute lower bound.
When $s_i=r_i$, the old components and the nonempty new part must
remain separate; when the degree is one, every edge is a component.
The single detachment attains all these bounds.

For a fixed color, an old set $S$, and $t\in[\tau]$, the added
edges with old set $S$ and $t$ new vertices have fixed total
multiplicity.
The detachment makes their multiplicities the floor or ceiling of
their average.
If their original multiplicities lie in $[L_i,U_i]$, these rounded
averages lie in the same integer interval.
Old edges are unchanged, proving the bound assertion.

For the convex assertion, if integer multiplicities $u,v$ satisfy
$u\ge v+2$, convexity gives
\[
 \varphi(u-1)+\varphi(v+1)\le\varphi(u)+\varphi(v).
\]
Repeatedly making this change ends in a list whose entries differ
by at most one: the sum of squares strictly decreases at each change.
Every such list with the fixed total has the same convex sum.
Fair detachment produces this distribution in every class just
described, simultaneously in all colors.
Adding the comparisons over the classes, with the old edges fixed,
proves the convex inequality for every $\varphi$ for the same $\cF'$.
Taking $L_i$ and $U_i$ to be the original extrema proves
Theorem~\ref{CCAII:thm:main}.
\end{proof}

\subsection{Added-edge counts and multiplicity spread}

In color $i$, put $e_i=(ns_i-mr_i)/h$, the required number of added edges.
This is an integer by admissibility and the regularity of $G_i$.
Let $x_{it}\in\mathbb N\cup\{0\}$ count the type-$t$ edges in color $i$.
Write $\mathscr X$ for the set of matrices
$X=(x_{it})_{i\in[k],\,t\in[\tau]}$ satisfying
\begin{equation}\label{CCAII:eq:count-system}
 \begin{cases}
  \displaystyle\sum\nolimits_{i\in[k]}x_{it}=\lambda C_t
      &\text{for }t\in[\tau],\\[3pt]
  \displaystyle\sum\nolimits_{t\in[\tau]}x_{it}=e_i
      &\text{for }i\in[k],\\[3pt]
  \displaystyle\sum\nolimits_{t\in[\tau]}t x_{it}=(n-m)s_i
      &\text{for }i\in[k].
 \end{cases}
\end{equation}
The three lines count edge copies of each type, added edges in each
color, and incidences with new vertices.

Put $g_i^-=\min\nolimits_{e\in E(K_m^h)}\mult_{G_i}(e)$ and
$g_i^+=\max\nolimits_{e\in E(K_m^h)}\mult_{G_i}(e)$.
For additional target colors, both are zero.
For $X\in\mathscr X$, define
\[
 \Delta_i(X)=
 \max\left\{g_i^+,\max\nolimits_{t\in[\tau]}
          \left\lceil\frac{x_{it}}{C_t}\right\rceil\right\}
 -
 \min\left\{g_i^-,\min\nolimits_{t\in[\tau]}
          \left\lfloor\frac{x_{it}}{C_t}\right\rfloor\right\}.
\]
For fixed counts, the old extrema and these type averages force
$\Delta_i(X)$ as a lower bound for the spread in color $i$.

\begin{theorem}[Criterion from added-edge counts]\label{CCAII:thm:new-edge-count-criterion}
An embedding exists if and only if admissibility holds and
$\mathscr X\ne\varnothing$.
Whenever these conditions hold, every $X\in\mathscr X$ is realized by an embedding satisfying
\[
 \operatorname{spr}(F_i)=\Delta_i(X),
 \qquad
 c(F_i)=\kappa_i
 \quad\text{for }i\in[k].
\]
Every embedding with added-edge counts $X$ has spread at least
$\Delta_i(X)$ in color $i$.
\end{theorem}

\begin{proof}
An embedding determines the counts in \eqref{CCAII:eq:count-system}
and satisfies admissibility.

Conversely, assume admissibility and fix $X\in\mathscr X$.
For $i\in[k]$, the equations give
\[
 \sum\nolimits_{t\in[\tau]}(h-t)x_{it}
 =he_i-(n-m)s_i=m(s_i-r_i).
\]
On the old vertices, form the nonuniform hypergraph
\[
 \mathcal Q=\bigcup\nolimits_{t\in[\tau],\,t<h}
             \lambda\binom{n-m}t K_m^{h-t}.
\]
Theorem~\ref{CCAII:thm:balanced-R} gives a factorization
$\{R_1,\dots,R_k\}$ in which $R_i$ has degree $s_i-r_i$,
has $x_{it}$ edges of size $h-t$ for $t<h$, and is equimultiple
within each size.
The preceding identity is precisely its incidence equation, while
$\sum\nolimits_{i\in[k]}x_{it}=\lambda C_t$ gives the required edge totals.
Degrees $s_i-r_i=0$ are allowed by that theorem.
The term $t=h$, when present, has no old incidences and is handled separately.

Retain every old edge of $G_i$, replace each size-$(h-t)$ edge
$S$ of $R_i$ by $S\alpha^t$, and add $x_{ih}$ copies of $\alpha^h$
when $\tau=h$.
Call the resulting colored hypergraph $\cH$.
Each old vertex has added-edge degree $s_i-r_i$ in color $i$,
and $\deg_{H_i}(\alpha)=(n-m)s_i$.
Since $t\le\tau\le n-m$, Lemma~\ref{CCAII:mainthconfdet} applies.

Split $\alpha$ into the $n-m$ new vertices.
The uncolored multiplicity of $S\alpha^t$ is
$\lambda\binom{n-m}t$, so multiplicity fairness restores
$\lambda$ copies of every added edge.
Degree fairness gives new-vertex degree $s_i$, while old-vertex
degrees are $r_i+(s_i-r_i)=s_i$.
Old edges remain fixed, and the added-edge counts remain $X$.
Hence the resulting $\cF$ is an embedding.

This detachment also has $c(F_i)=\kappa_i$.
Its amalgamated color classes have the prescribed old edges, added-edge
degree $s_i-r_i$ at each old vertex, and degree $(n-m)s_i$ at $\alpha$,
exactly the properties used in the component analysis of
Lemma~\ref{CCAII:lem:component-optimal}.
If $s_i=r_i>0$, there are no mixed edges in that color, so positive degree at $\alpha$
forces all-new edges to be present; the degree-one case is forced.
Hence the component analysis in Lemma~\ref{CCAII:lem:component-optimal}
applies directly to this $\cH$.

Put $P_t=\binom m{h-t}$ and $Q_t=\binom{n-m}t$, so
$C_t=P_tQ_t$ and both denominators are positive.
Before detachment, the multiplicity of $S\alpha^t$ in color $i$
is either $\lfloor x_{it}/P_t\rfloor$ or $\lceil x_{it}/P_t\rceil$.
For $t=h$, this is simply $x_{ih}$, with $P_h=1$ and $S=\varnothing$.
After detachment, the nested floor--ceiling identities give every
type-$t$ multiplicity between
$\lfloor x_{it}/C_t\rfloor$ and $\lceil x_{it}/C_t\rceil$.
If the average is not an integer, both endpoints occur because
the $C_t$ multiplicities sum to $x_{it}$; if it is an integer,
every multiplicity equals it.
Together with the fixed old extrema, this proves
$\operatorname{spr}(F_i)=\Delta_i(X)$.

For any embedding with the same counts, the type-$t$ multiplicities
still sum to $x_{it}$.
At least one is at most $\lfloor x_{it}/C_t\rfloor$ and at least one
is at least $\lceil x_{it}/C_t\rceil$.
The old extrema are also fixed, so its spread is at least $\Delta_i(X)$.
\end{proof}

Assume $\mathscr E\ne\varnothing$.
Both $\mathscr E$ and $\mathscr X$ are finite; in particular,
$0\le x_{it}\le\lambda C_t$.
For $i\in[k]$, the individual minimum spread is
\[
 \sigma_i=
 \min\nolimits_{\cF\in\mathscr E}\operatorname{spr}(F_i)
 =
 \min\nolimits_{X\in\mathscr X}\Delta_i(X).
\]
The values $\sigma_i$ need not be attained simultaneously, since
the column-sum equations link the colors.
If $\Phi:(\mathbb N\cup\{0\})^k\to\mathbb R$ is nondecreasing
in each coordinate, then
\[
 \min\nolimits_{\cF\in\mathscr E}
   \Phi\bigl(\operatorname{spr}(F_1),\dots,\operatorname{spr}(F_k)\bigr)
 =
 \min\nolimits_{X\in\mathscr X}
   \Phi\bigl(\Delta_1(X),\dots,\Delta_k(X)\bigr),
\]
and one minimizing embedding has $c(F_i)=\kappa_i$ for every color.
This includes the maximum over any nonempty set of colors and
every nonnegative weighted sum.
Necessity follows from the lower bounds $\Delta_i(X)$, and
sufficiency from their simultaneous attainment in
Theorem~\ref{CCAII:thm:new-edge-count-criterion}.

For prescribed $d_1,\dots,d_k\in\mathbb N\cup\{0\}$, an embedding
with $\operatorname{spr}(F_i)\le d_i$ for $i\in[k]$ exists exactly
when admissibility holds and some $X\in\mathscr X$ satisfies
$\Delta_i(X)\le d_i$ for $i\in[k]$.
Theorem~\ref{CCAII:thm:new-edge-count-criterion} gives such an embedding
while minimizing all component counts.

Existence therefore depends on the source degrees, while the spread
formulas use the old multiplicities only through $g_i^-$ and $g_i^+$.
For $s_i=r_i$, the component minimum also uses $c(G_i)$.

For fixed $X$, the construction minimizes every convex sum as well.
Write $x_{it}=C_ta_{it}+\rho_{it}$ with
$a_{it},\rho_{it}\in\mathbb N\cup\{0\}$ and $0\le\rho_{it}<C_t$.
The balancing argument in Lemma~\ref{CCAII:lem:component-optimal}
shows that the minimum convex sum in color $i$ is
\[
 \sum\nolimits_{e\in E(K_m^h)}
       \varphi\bigl(\mult_{G_i}(e)\bigr)
 +
 \sum\nolimits_{t\in[\tau]}
       \bigl((C_t-\rho_{it})\varphi(a_{it})
                   +\rho_{it}\varphi(a_{it}+1)\bigr).
\]
Theorem~\ref{CCAII:thm:new-edge-count-criterion} attains this value
for every color and every convex $\varphi:\mathbb R\to\mathbb R$
in the same embedding, with all component counts minimum.

\subsection{Equimultiple embeddings}

\begin{corollary}\label{CCAII:cor:equimultiple-new-edge-count}
Let $I\subseteq[k]$.
An embedding with equimultiple $F_i$ for $i\in I$ exists if and
only if admissibility holds and some $X\in\mathscr X$ satisfies
$\Delta_i(X)\le1$ for $i\in I$.
Equivalently, the last conditions are
\[
 \mult_{G_i}(e)\approx s_i/B
       \quad\text{for }i\in I,\ e\in E(K_m^h),
 \qquad
 x_{it}/C_t\approx s_i/B
       \quad\text{for }i\in I,\ t\in[\tau].
\]
Whenever these conditions hold, the embedding may also be chosen
with $c(F_i)=\kappa_i$ for $i\in[k]$.
\end{corollary}

\begin{proof}
The first criterion is the spread-bound criterion with bound one.
The average underlying-edge multiplicity of an $s_i$-factor is
$(ns_i/h)/\binom nh=s_i/B$.
Integer multiplicities of spread at most one must therefore lie
between $\lfloor s_i/B\rfloor$ and $\lceil s_i/B\rceil$.
This gives the conditions on the old edges and on each type average.

Conversely, these conditions put the old multiplicities in that
integer interval.
For each type, the floor and ceiling of $x_{it}/C_t$ also lie
in the interval.
Theorem~\ref{CCAII:thm:new-edge-count-criterion} gives an embedding with these
multiplicities and all component minima together.
If $s_i/B$ is an integer, the interval has one value, so every
multiplicity in that color is constant.
\end{proof}

\section{Exact criteria for \texorpdfstring{$h=2,3$}{h=2,3}}\label{CCAII:sec:h23-main}

\subsection{The graph case}

For $h=2$, the second and third equations in \eqref{CCAII:eq:count-system}
force all added-edge counts.  Put $d=n-m$.  If $d=1$, set
$a_{i1}=s_i/m$.  If $d\ge2$, set
\[
 a_{i1}=\frac{s_i-r_i}{d},\qquad
 a_{i2}=\frac{mr_i-(2m-n)s_i}{d(d-1)}.
\]

\begin{theorem}\label{CCAII:embedrmsn2}
An $\vb r$-factorization of $\lambda K_m$ embeds in an
$\vb s$-factorization of $\lambda K_n$ satisfying $c(F_i)=\kappa_i$ and
\[
 \operatorname{spr}(F_i)=\sigma_i=
 \begin{cases}
  \max\{g_i^+,\lceil a_{i1}\rceil\}
  -\min\{g_i^-,\lfloor a_{i1}\rfloor\},&d=1,\\[4pt]
  \max\{g_i^+,\lceil a_{i1}\rceil,\lceil a_{i2}\rceil\}
  -\min\{g_i^-,\lfloor a_{i1}\rfloor,\lfloor a_{i2}\rfloor\},&d\ge2,
 \end{cases}
 \qquad i\in[k],
\]
if and only if admissibility holds and $mr_i\ge(2m-n)s_i$ for $i\in[k]$.
\end{theorem}

\begin{proof}
\begin{enumerate}[label=\textup{(\roman*)},leftmargin=1.7em,labelsep=.35em,itemsep=3pt,topsep=2pt]
\item If $n=m+1$, the condition $mr_i\ge(2m-n)s_i$ becomes
$mr_i\ge(m-1)s_i$.  Summing over
the colors gives equality, since
$m\sum\nolimits_{i\in[k]}r_i=(m-1)\sum\nolimits_{i\in[k]}s_i=\lambda m(m-1)$.
Hence $mr_i=(m-1)s_i$ for every $i$, and
\[
 x_{i1}=s_i=\frac{ns_i-mr_i}{2}.
\]
Thus $x_{i1}$ is uniquely determined, and $C_1=m$, so
Theorem~\ref{CCAII:thm:new-edge-count-criterion} gives the stated value of
$\sigma_i$.

\item If $n\ge m+2$, types $1$ and $2$ occur, and the second and third equations in
\eqref{CCAII:eq:count-system} give uniquely
\[
 x_{i1}=m(s_i-r_i),\qquad
 x_{i2}=\frac{mr_i-(2m-n)s_i}{2}.
\]
The standing inequality $r_i\le s_i$ gives $x_{i1}\ge0$, while
$x_{i2}\ge0$ if and only if $mr_i\ge(2m-n)s_i$.  These quantities are
integers: $2\mid mr_i$ and $2\mid ns_i$, so
$mr_i-(2m-n)s_i=mr_i-2ms_i+ns_i$ is even.  Summing over the colors and using
$\sum\nolimits_{i\in[k]}r_i=\lambda(m-1)$ and
$\sum\nolimits_{i\in[k]}s_i=\lambda(n-1)$ gives the required type totals.
Since $X$ is unique, Theorem~\ref{CCAII:thm:new-edge-count-criterion} gives
$\operatorname{spr}(F_i)=\sigma_i$; substituting these two counts into
$\Delta_i(X)$ gives the stated formula.
\end{enumerate}
\end{proof}

\begin{corollary}\label{CCAII:cor:h2-equimultiple}
Let $I\subseteq[k]$.  An $\vb r$-factorization of $\lambda K_m$
embeds in an $\vb s$-factorization of $\lambda K_n$
in which $c(F_i)=\kappa_i$ and $\operatorname{spr}(F_i)=\sigma_i$ for
$i\in[k]$, and $F_i$ is equimultiple for $i\in I$, if and only if
admissibility holds and
\[
\begin{aligned}
 mr_i&\ge(2m-n)s_i
 &&\text{for }i\in[k],\\
 \mult_{G_i}(e)&\approx\frac{s_i}{n-1}
 &&\text{for }i\in I,\ e\in E(K_m),\\
 \frac{s_i-r_i}{n-m}&\approx\frac{s_i}{n-1},\qquad
 \frac{mr_i-(2m-n)s_i}{(n-m)(n-m-1)}\approx\frac{s_i}{n-1}
 &&\text{for }i\in I,\ n\ge m+2.
\end{aligned}
\]
\end{corollary}

\begin{remark}
Let $I\subseteq[k]$, and let $i\in I\setminus[q]$ be an additional target color.
Suppose $n\ge m+2$.  The conditions in Corollary~\ref{CCAII:cor:h2-equimultiple}
involving color $i$ reduce to $n\ge2m$ and $s_i\le n-m$.
Indeed, for $r_i=0$ the existence condition $mr_i\ge(2m-n)s_i$ is exactly
$n\ge2m$.  Under this inequality, the old-edge and type-$1$ equimultiplicity
conditions together are equivalent
to $s_i\le n-m$, and the type-$2$ condition then follows.  Thus these are
exactly the conditions for a target color absent from the source; when
they hold, that equimultiple target factor is simple.
\end{remark}

\begin{remark}
In the classical simple-graph case $\lambda=1$ with constant source and target
degrees $r$ and $s$, we have $q=(m-1)/r$ and $k=(n-1)/s$.
Theorem~\ref{CCAII:embedrmsn2} reduces to the following: an embedding exists if
and only if admissibility holds and $k>q$ implies $n\ge2m$.
Indeed, when $k=q$, the relation $r(n-1)=s(m-1)$ makes the source-color
inequality automatic.  When $k>q$, applying the exact criterion to an
additional target color gives $n\ge2m$, which also makes the source-color
inequality automatic.  Whenever an embedding exists, all
target factors can be chosen connected if and only if $s>r$.  Since
$\lambda=1$, all factors are simple.  For $n\ge m+2$, this is the connected
embedding criterion of Hilton--Johnson--Rodger--Wantland~\cite{HJRW2003}.
\end{remark}

\subsection{The \texorpdfstring{$3$}{3}-uniform case}

For $h=3$, the count equations of Section~\ref{CCAII:sec:regular-embedding}
have only one degree of freedom in each color.  Put $d=n-m$ and
\[
 A=\binom{m-1}{2},\quad B=\binom{n-1}{2},\quad
 C_1=d\binom m2,\quad C_2=m\binom d2,\quad C_3=\binom d3.
\]
When $d\ge3$, put
\[
 \delta_i=s_i-r_i,\qquad
 \beta_i=\frac{(3m-n)s_i-2mr_i}{3},\qquad
 \ell_i=\max\{0,\beta_i\},\qquad
 u_i=\left\lfloor\frac{m\delta_i}{2}\right\rfloor
 \quad\text{for }i\in[k].
\]
The number $\beta_i$ is an integer because $3\mid mr_i$ and
$3\mid ns_i$.  If $x_i=x_{i1}$, the other two counts are
\[
 x_{i2}=m\delta_i-2x_i,\qquad x_{i3}=x_i-\beta_i.
\]
Thus the embedding problem reduces to choosing one integer from a prescribed
interval in each color with the required sum.  For
$x\in[\ell_i,u_i]\cap\mathbb Z$, define
\[
\begin{aligned}
 \Delta_i(x)
 &=\max\left\{g_i^+,\left\lceil\frac{x}{C_1}\right\rceil,
 \left\lceil\frac{m\delta_i-2x}{C_2}\right\rceil,
 \left\lceil\frac{x-\beta_i}{C_3}\right\rceil\right\}\\
 &\quad-\min\left\{g_i^-,\left\lfloor\frac{x}{C_1}\right\rfloor,
 \left\lfloor\frac{m\delta_i-2x}{C_2}\right\rfloor,
 \left\lfloor\frac{x-\beta_i}{C_3}\right\rfloor\right\}.
\end{aligned}
\]
\begin{theorem}\label{CCAII:thm:h3-exact}
An $\vb r$-factorization of $\lambda K_m^3$ embeds in an
$\vb s$-factorization of $\lambda K_n^3$ satisfying
$c(F_i)=\kappa_i$ for $i\in[k]$ if and only if admissibility holds and
\[
\begin{aligned}
 mr_i&\ge(3m-2n)s_i
 &&\text{for }i\in[k],\\
 mr_i&\le(m-1)s_i
 &&\text{for }i\in[k],\ n=m+2,\\
 \sum\nolimits_{i\in[k]}\ell_i&\le\lambda C_1
 &&\text{for }n\ge m+3.
\end{aligned}
\]
If $n\in\{m+1,m+2\}$, such an embedding can be chosen with
$\operatorname{spr}(F_i)=\sigma_i$ for $i\in[k]$.  If $n\ge m+3$, then
for every choice of integers $x_1,\dots,x_k$ satisfying
\[
 \ell_i\le x_i\le u_i\quad\text{for }i\in[k],\qquad
 \sum\nolimits_{i\in[k]}x_i=\lambda C_1,
\]
there is such an embedding satisfying
$\operatorname{spr}(F_i)=\Delta_i(x_i)$ for $i\in[k]$.
\end{theorem}

\begin{proof}
\begin{enumerate}[label=\textup{(\roman*)},leftmargin=1.7em,labelsep=.35em,itemsep=3pt,topsep=2pt]
\item If $n=m+1$, the condition $mr_i\ge(3m-2n)s_i$ becomes
$mr_i\ge(m-2)s_i$.  Summing over
the colors gives equality, since
$m\sum\nolimits_{i\in[k]}r_i=(m-2)\sum\nolimits_{i\in[k]}s_i$.
Thus $mr_i=(m-2)s_i$ for every $i$, and
$x_{i1}=s_i=(ns_i-mr_i)/3$.  Admissibility gives
$\sum\nolimits_{i\in[k]}x_{i1}=\lambda C_1$.  Thus $x_{i1}$ is uniquely determined,
so Theorem~\ref{CCAII:thm:new-edge-count-criterion} gives an embedding with
$\operatorname{spr}(F_i)=\sigma_i$ and $c(F_i)=\kappa_i$ for every $i$.

\item If $n=m+2$, the two inequalities in the theorem become
$(m-4)s_i\le mr_i\le(m-1)s_i$.  The second and third equations in
\eqref{CCAII:eq:count-system} uniquely determine
\[
 x_{i1}=\frac{2((m-1)s_i-mr_i)}3,
 \qquad
 x_{i2}=\frac{mr_i-(m-4)s_i}{3}.
\]
These two counts are nonnegative exactly under these inequalities.  Both quantities are integers, since $3\mid mr_i$ and
$3\mid(m+2)s_i$, while $m-1\equiv m-4\equiv m+2\pmod3$.
Summing these formulas and using $\sum\nolimits_{i\in[k]}r_i=\lambda A$ and
$\sum\nolimits_{i\in[k]}s_i=\lambda B$ yields the required type totals.  Theorem~\ref{CCAII:thm:new-edge-count-criterion}
therefore applies, and the resulting embedding has
$\operatorname{spr}(F_i)=\sigma_i$ and $c(F_i)=\kappa_i$ for every $i$.

\item If $n\ge m+3$, put $x=x_{i1}$.  The second and third equations in \eqref{CCAII:eq:count-system} force
\[
 x_{i2}=m(s_i-r_i)-2x,
 \qquad
 x_{i3}=x+\frac{(n-3m)s_i+2mr_i}{3}.
\]
Thus the nonnegative counts are exactly those obtained by choosing an integer
$x\in[\ell_i,u_i]$.  This interval is nonempty exactly when
\[
 mr_i\ge(3m-2n)s_i.
\]
Indeed, when $\ell_i=0$ this inequality follows from
$(3m-n)s_i\le2mr_i$, since $n>m$; when $\ell_i>0$,
the inequality $\ell_i\le u_i$ simplifies to precisely
$mr_i\ge(3m-2n)s_i$.
Because the intervals have integer endpoints, there are integers $x_i$ with
$\ell_i\le x_i\le u_i$ and $\sum\nolimits_{i\in[k]}x_i=\lambda C_1$ if and only if
\[
 \sum\nolimits_{i\in[k]}\ell_i\le\lambda C_1\le\sum\nolimits_{i\in[k]}u_i.
\]
The upper inequality is automatic.  Put $\delta_i=s_i-r_i$.  Since $\delta_i\in\mathbb N\cup\{0\}$,
\[
 u_i=\left\lfloor\frac{m\delta_i}{2}\right\rfloor
 \ge\frac{(m-1)\delta_i}{2}.
\]
Admissibility and the source factorization give
$\sum\nolimits_{i\in[k]}\delta_i=\lambda(B-A)$.  Since $n-m\ge3$,
\[
 \sum\nolimits_{i\in[k]} u_i
 \ge\lambda\frac{m-1}{2}(B-A)
 =\lambda C_1+\frac{\lambda(n-m)(m-1)(n-m-3)}4
 \ge\lambda C_1.
\]
Thus only the lower inequality is needed.  Once
$\sum\nolimits_{i\in[k]}x_{i1}=\lambda C_1$ holds, summing the second and third
equations in \eqref{CCAII:eq:count-system} over the colors forces
$\sum\nolimits_{i\in[k]}x_{i2}=\lambda C_2$ and
$\sum\nolimits_{i\in[k]}x_{i3}=\lambda C_3$.  These choices give every possible set of counts.  Theorem~\ref{CCAII:thm:new-edge-count-criterion}
then gives the stated spread and component counts.
\end{enumerate}
\end{proof}

\begin{corollary}\label{CCAII:cor:h23-large-order}
Suppose $h\in\{2,3\}$ and $n\ge2m$.  An $\vb r$-factorization of
$\lambda K_m^h$ embeds in an $\vb s$-factorization of $\lambda K_n^h$
satisfying $c(F_i)=\kappa_i$ for $i\in[k]$ if and only if admissibility
holds.  Whenever admissibility holds, every $X\in\mathscr X$ is realized by
such an embedding with $\operatorname{spr}(F_i)=\Delta_i(X)$ for $i\in[k]$.
\end{corollary}

\begin{proof}
For $h=2$, Theorem~\ref{CCAII:embedrmsn2} applies with $n\ge m+2$, and
$mr_i\ge(2m-n)s_i$ is automatic when $n\ge2m$.

For $h=3$, put $d=n-m\ge m$.  Then $n\ge m+3$, so the last case of
Theorem~\ref{CCAII:thm:h3-exact} applies.  The inequality
$mr_i\ge(3m-2n)s_i$ is automatic because $3m-2n=m-2d\le0$.  Put $\delta_i=s_i-r_i$.  Since
$B-A=d(2m+d-3)/2$ and $C_1=d\binom m2$, we have
\[
 \ell_i
 =\max\left\{0,\frac{(2m-d)\delta_i-dr_i}{3}\right\}.
\]
If $d\ge2m$, then $\ell_i=0$ for $i\in[k]$.  If $m\le d<2m$, then
\[
 \sum\nolimits_{i\in[k]}\ell_i
 \le \frac{2m-d}{3}\sum\nolimits_{i\in[k]}\delta_i
 =\lambda\frac{2m-d}{3}(B-A)
 \le\lambda C_1,
\]
because
\[
 6C_1-2(2m-d)(B-A)=d(d-m)(d+m-3)\ge0.
\]
Thus the condition $\sum\nolimits_{i\in[k]}\ell_i\le\lambda C_1$ in
Theorem~\ref{CCAII:thm:h3-exact} is automatic, and the result follows.
\end{proof}

Unlike the graph case, the values $\sigma_i$ need not be simultaneously
attainable when $h=3$.  For example, take
$m=6$, $n=10$, $\lambda=4$, and an equimultiple source factorization with
\[
 \vb r=(10,15,15)^{\mathsf T},\qquad
 \vb s=(75,51,18)^{\mathsf T};
\]
such a source factorization exists by
Theorem~\ref{CCAII:thm:balanced-R}.  Here $(C_1,C_2,C_3)=(60,36,4)$,
and the minimum and maximum old-edge multiplicities in the three colors are
$(1,1)$, $(1,2)$, and $(1,2)$, respectively.  Writing $x_i=x_{i1}$, the
counts are
\[
\begin{aligned}
 (x_{11},x_{12},x_{13})&=(x_1,390-2x_1,x_1-160),\\
 (x_{21},x_{22},x_{23})&=(x_2,216-2x_2,x_2-76),\\
 (x_{31},x_{32},x_{33})&=(x_3,18-2x_3,x_3+12),
\end{aligned}
\]
where $160\le x_1\le195$, $76\le x_2\le108$, and $0\le x_3\le9$.
Substitution in $\Delta_i$ gives
$(\sigma_1,\sigma_2,\sigma_3)=(2,1,3)$, attained respectively when
\[
 164\le x_1\le172,\qquad 80\le x_2\le84,\qquad x_3=0.
\]
Each coordinate minimum occurs in some embedding with the required
type-$1$ total: for instance, $(x_1,x_2,x_3)=(164,76,0)$ attains the minima
in colors $1$ and $3$, while $(160,80,0)$ attains the minimum in color $2$.  However,
$\lambda C_1=240$, whereas simultaneous attainment of all three minima would
require
\[
 x_1+x_2+x_3\ge164+80=244.
\]
Hence no single embedding attains all three coordinatewise minima; the
colors must be optimized jointly in Theorem~\ref{CCAII:thm:h3-exact}.

Equimultiplicity only narrows these intervals.  Let
$I\subseteq[k]$.  An embedding with $c(F_i)=\kappa_i$ for $i\in[k]$ and
with $F_i$ equimultiple for $i\in I$ exists exactly when the conditions of
Theorem~\ref{CCAII:thm:h3-exact} hold and the counts can also be chosen
so that these colors have spread at most one.  Thus, if
$n\in\{m+1,m+2\}$, the unique $X$ must satisfy $\Delta_i(X)\le1$ for
$i\in I$.  If $d\ge3$, there must be integers $x_1,\dots,x_k$ with the
required type-$1$ total and with $\Delta_i(x_i)\le1$ for $i\in I$.
This is exactly Corollary~\ref{CCAII:cor:equimultiple-new-edge-count}
applied to the count description in Theorem~\ref{CCAII:thm:h3-exact}.

To make the integer endpoints explicit, put
$a_i=\lfloor s_i/B\rfloor$ and $b_i=\lceil s_i/B\rceil$ for $i\in I$.
After requiring the prescribed old-edge multiplicities in color $i$ to lie
in $[a_i,b_i]$, equimultiplicity is equivalent to
$L_i\le x_i\le U_i$, where
\[
 L_i=\max\left\{a_iC_1,
 \left\lceil\frac{m\delta_i-b_iC_2}{2}\right\rceil,
 a_iC_3+\beta_i\right\},
\]
\[
 U_i=\min\left\{b_iC_1,
 \left\lfloor\frac{m\delta_i-a_iC_2}{2}\right\rfloor,
 b_iC_3+\beta_i\right\}.
\]
Indeed, these are exactly the bounds $a_iC_t\le x_{it}\le b_iC_t$
after substituting the three counts above; they already imply
nonnegativity, so separate copies of $\ell_i$ and $u_i$ are redundant.
For $i\notin I$, put $L_i=\ell_i$ and $U_i=u_i$.  The simultaneous
condition is precisely
\[
 L_i\le U_i\quad\text{for }i\in[k],\qquad
 \sum\nolimits_{i\in[k]}L_i\le\lambda C_1
 \le\sum\nolimits_{i\in[k]}U_i.
\]
Because all endpoints are integers, these inequalities are also sufficient:
choose $x_i\in[L_i,U_i]\cap\mathbb Z$ with the required sum.  The count
equations then force the remaining column totals.

\subsection{Constant degrees}
Put
\[
 D=B-A,\qquad q=\lambda A/r,\qquad k=\lambda B/s.
\]
The following lemma gives the simple-target criterion used for
equimultiplicity.

\begin{lemma}\label{CCAII:lem:scalar-simple-counts}
Suppose a simple $r$-factorization of $\lambda K_m^3$ embeds in an
$s$-factorization of $\lambda K_n^3$, where $s<B$.  Then it has an
embedding with all target factors simple if and only if, when $k>q$,
\[
 s\le D,\qquad (3m-n)s\le3C_1.
\]
\end{lemma}

\begin{proof}
Simplicity gives $1\le r\le A$, so $q\ge\lambda$; also $q\le k$ gives
$rB\ge sA$.  In an additional target color, each old vertex needs $s$ incident added
edges, but only $D$ are available, so $s\le D$.  The count equations give
$x_{i1}\ge(3m-n)s/3$ in such a color, so $x_{i1}\le C_1$ gives the
second inequality.  This proves necessity for $k>q$.

For sufficiency, first suppose $d\le2$.  If $k=q$, the counts are
forced and equal in every color, so $x_{it}=sC_t/B<C_t$.  If $k>q$,
existence excludes $d=1$.  For $d=2$, the condition $2n\ge3m$ forces
$m=3$ or $4$.  When $m=3$, simplicity and divisibility give $r=1$ and
$(3m-n)s\le3C_1$ gives $s=3$; the source and additional count rows are
$(2,2)$ and $(4,1)$, within $(C_1,C_2)=(6,3)$.  When $m=4$, simplicity
and divisibility give $r=3$, while existence and the assumed bounds give
$4\le s\le6$; the rows are $(2s-8,4)$ and $(2s,0)$, within
$(C_1,C_2)=(12,4)$.

Now suppose $d\ge3$.  For $r_i\in\{r,0\}$, put
$\delta_i=s-r_i$ and $\beta_i=((3m-n)s-2mr_i)/3\in\mathbb Z$.  The bounds $0\le x_{it}\le C_t$
correspond exactly to integers $x_i=x_{i1}$ in
\[
 L_i=\max\left\{0,\beta_i,
 \left\lceil\frac{m\delta_i-C_2}{2}\right\rceil\right\},\qquad
 U_i=\min\left\{C_1,\left\lfloor\frac{m\delta_i}{2}\right\rfloor,
 C_3+\beta_i\right\}.
\]
It is enough to find real numbers $y_i\in[L_i,U_i]$ with
$\sum\nolimits_{i\in[k]}y_i=\lambda C_1$.  Then
$\sum\nolimits_{i\in[k]}L_i\le\lambda C_1\le\sum\nolimits_{i\in[k]}U_i$, so the integer intervals contain integers $x_i$ with the
required sum.

We use $mD=2C_1+C_2$ and
$C_2-D=d(m-1)(d-3)/2\ge0$.  For every integer $0\le z\le D$,
\[
 \left\lceil\frac{mz-C_2}{2}\right\rceil
 \le\frac{zC_1}{D}\le
 \left\lfloor\frac{mz}{2}\right\rfloor.
\]
The gap on the left before rounding is $C_2(D-z)/(2D)$, at least $1/2$
unless $z=D$; the gap on the right is $zC_2/(2D)$, at least $1/2$ unless
$z=0$.  Since the expressions being rounded are integers or half-integers,
and the endpoint equalities are exact, the displayed inequalities follow.

\begin{enumerate}[label=\textup{(\roman*)},leftmargin=1.7em,labelsep=.35em,itemsep=3pt,topsep=2pt]
\item Suppose $k=q$.  Put $\alpha=s/B=r/A$, and choose
$y_i=\alpha C_1$ in every color.  Here $\delta_i=\alpha D$ and
$\beta_i=\alpha(C_1-C_3)$, since
$(3m-n)B-2mA=3(C_1-C_3)$.  Thus
$\beta_i\le y_i\le C_3+\beta_i$, and the other bounds follow from the
rounding inequality.  Finally $k\alpha C_1=\lambda C_1$.

\item Suppose $k>q$ and $d\ge m$.  The hypotheses give $s\le D$ and
$\beta_0=(3m-n)s/3\le C_1$.  Choose $y_i=\delta_iC_1/D$.  Put
\[
 K=3C_1-(3m-n)D=\frac{d(d-m)(d+m-3)}2=3C_3-dA\ge0.
\]
Direct subtraction gives
\[
 y_i-\beta_i=\frac{K\delta_i/D+dr_i}{3}\ge0,\qquad
 C_3+\beta_i-y_i=
 \frac{K(1-\delta_i/D)+d(A-r_i)}3\ge0.
\]
The remaining bounds follow from the rounding inequality, and
$\sum\nolimits_{i\in[k]}\delta_i=\lambda D$ gives the required total.

\item Suppose $k>q$ and $m/2\le d<m$.  The hypotheses give
$s\le D$ and $\beta_0=(3m-n)s/3\le C_1$, while existence gives
$2d\ge m$.  Write $\beta_r$ for a source
color and $\beta_0$ for an additional target color, and put $P=2d-m\ge0$.
Choose
\[
 y_r=\beta_r+\frac{rC_3}{A},\qquad y_0=\beta_0.
\]
These values have the required total, since
\[
 q\beta_r+(k-q)\beta_0=\lambda(C_1-C_3),\qquad qr/A=\lambda.
\]
Moreover, existence gives $(k-q)\beta_0\le\lambda C_1$, which is
equivalent to $y_r\ge0$ by the preceding identities.

For an additional target color, $PD\le3C_2$ and $s\le D$ give
$\lceil(ms-C_2)/2\rceil\le\beta_0$, so $L_0=\beta_0=y_0$.
Also $y_0\le C_1$, $ms/2-\beta_0=Ps/6\ge0$, and
$C_3+\beta_0\ge\beta_0$.  Hence $y_0\in[L_0,U_0]$.

For a source color, $y_r\ge\beta_r$ and
$y_r\le\beta_r+C_3$ because $r\le A$.  Put
$\gamma=6C_3/A-m$.  Since $d\le m-1$,
\[
 P-\gamma=4+
 \frac{2(d-2)((m-1)(m-2)-d(d-1))}{(m-1)(m-2)}\ge4.
\]
Thus
\[
 \frac{m(s-r)}2-y_r
 =\frac{P(s-r)+(P-\gamma)r}{6}\ge\frac23,
\]
so $y_r\le\lfloor m(s-r)/2\rfloor$.  Also
$y_r=\beta_0-r(2m/3-C_3/A)\le\beta_0\le C_1$, since
$C_3/A\le(m-3)/3$.

It remains only to check the lower ceiling.  The gap before rounding is
\[
 y_r-\frac{m(s-r)-C_2}{2}
 =\frac{3C_2-Ps+\gamma r}{6}.
\]
If $\gamma\ge0$, use $r\ge sA/B$ and $PB-\gamma A=3C_2$ to bound this
from below by $C_2(B-s)/(2B)\ge1/2$.  Here $B-s\ge A$ and
$C_2A\ge B$, because $C_2\ge D$, $A\ge3$, and $D\ge3m\ge12$.
If $\gamma<0$ and $P>0$, use $r\le A$, $s\le D$, and
$3C_2-PD+\gamma A=PA$ to obtain a gap at least $PA/6\ge1/2$.
If $P=0$, the gap is $(-\gamma)(A-r)/6$, at least $2/3$ when $r<A$;
when $r=A$, it is zero and $y_r=\beta_r+C_3$ is an integer.
Thus the ceiling is at most $y_r$, and $y_r\in[L_r,U_r]$.
\end{enumerate}

In every case, the chosen $y_i$ lie in the required intervals and sum to
$\lambda C_1$.  We may therefore choose integers $x_i$ in the same
intervals with this sum; the count equations force the other two columns.  Their entries are at most
$C_t$, so Theorem~\ref{CCAII:thm:new-edge-count-criterion} gives simple
target factors with these counts.
\end{proof}

\begin{theorem}\label{CCAII:thm:h3-scalar}
An $r$-factorization of $\lambda K_m^3$ embeds in an $s$-factorization of
$\lambda K_n^3$ if and only if admissibility holds and, when $k>q$,
\[
 2n\ge3m,\qquad
 3r(n-m)\binom m2\ge(3m-n)(rB-sA).
\]
Under these conditions, an equimultiple embedding exists if and only if
\[
 \mult_{G_i}(e)\approx s/B\quad(i\in[q],\ e\in E(K_m^3)),\qquad
 \begin{cases}
 (3m-n)s\le3(n-m)\binom m2,& n<2m,\\[3pt]
 s\le B-A,& n\ge2m
 \end{cases}
 \quad\text{if }k>q.
\]
\end{theorem}

Whenever the existence conditions hold, an embedding can be chosen with
$c(F_i)=\kappa_i$ for $i\in[k]$; all target factors can be chosen connected
if and only if $s>r$.  Under the equimultiplicity conditions, the target
factorization is simple if $k>q$.

\begin{proof}
Since $qr=\lambda A$ and $ks=\lambda B$, we have $k=q$ if and only if
$rB=sA$, while $k-q=\lambda(rB-sA)/(rs)$.  We first verify the existence criterion from
Theorem~\ref{CCAII:thm:h3-exact}.
\begin{enumerate}[label=\textup{(\roman*)},leftmargin=1.7em,labelsep=.35em,itemsep=3pt,topsep=2pt]
\item If $n=m+1$, then $2n<3m$, so an embedding requires $k=q$:
an additional target color would violate $mr_i\ge(m-2)s$ because $r_i=0$.
The identity $rB=sA$ is exactly $mr=(m-2)s$.

\item If $n=m+2$, Theorem~\ref{CCAII:thm:h3-exact} requires
$(m-4)s\le mr\le(m-1)s$.  For $k>q$, the condition $2n\ge3m$ is
equivalent to $m\le4$, so $(m-4)s\le mr$ is automatic, while
$3r(n-m)\binom m2\ge(3m-n)(rB-sA)$ is equivalent to
$mr\le(m-1)s$.  For $k=q$, the identity $rB=sA$ implies both inequalities.

\item Suppose $n\ge m+3$ and put $d=n-m$ and $\theta=3m-n$.
If $k>q$, an additional target color has $r_i=0$, so its interval is
nonempty exactly when $2n\ge3m$; under this inequality the source-color
intervals are automatically nonempty.  If $k=q$, the identity $rB=sA$
instead gives
\[
 mA-(3m-2n)B
 =\frac{d(d-1)(2d+3m-4)}2\ge0,
\]
so the source-color intervals are again nonempty.

It remains to check $\sum\nolimits_{i\in[k]}\ell_i\le\lambda C_1$.  If $\theta\le0$, all
$\ell_i$ vanish.  Suppose $\theta>0$.  When $k=q$, either $\ell_i=0$ for every source color, or
\[
 \sum\nolimits_{i\in[k]}\ell_i=\lambda\frac{\theta B-2mA}{3}
 \le\lambda C_1,
\]
because
\[
 3C_1-(\theta B-2mA)=\frac{d(d-1)(d-2)}2\ge0.
\]
Now suppose $k>q$.  Since $k-q=\lambda(rB-sA)/(rs)$, if $\ell_i=0$
for every source color, then
\[
 \sum\nolimits_{i\in[k]}\ell_i
 =\lambda\frac{\theta(rB-sA)}{3r},
\]
so $\sum\nolimits_{i\in[k]}\ell_i\le\lambda C_1$ is exactly
\[
 3r(n-m)\binom m2\ge(3m-n)(rB-sA).
\]
If $\ell_i>0$ for the source colors, then
\[
 \sum\nolimits_{i\in[k]}\ell_i=\lambda\frac{\theta B-2mA}{3}\le\lambda C_1
\]
by the same identity, and
$3r(n-m)\binom m2\ge(3m-n)(rB-sA)$ is automatic as well, since
$\theta s>2mr$ implies
\[
 \theta(rB-sA)<r(\theta B-2mA)\le3rC_1.
\]
Thus the interval nonemptiness condition is exactly $2n\ge3m$ when
$k>q$, and the remaining sum condition is exactly
$3r(n-m)\binom m2\ge(3m-n)(rB-sA)$.
\end{enumerate}
This proves the existence criterion.  The formula for $\kappa_i$ gives the
connectedness assertion.

For equimultiplicity, the condition on old multiplicities is necessary
because an embedding preserves the multiplicities on old edges.  If $k>q$,
an additional target color has multiplicity zero on every old edge, so every
equimultiple target factor is simple and $s<B$.
Lemma~\ref{CCAII:lem:scalar-simple-counts} then gives the necessary
inequalities $s\le B-A$ and $(3m-n)s\le3C_1$, and hence the stated
piecewise condition.

Conversely, assume the condition on old multiplicities and, when $k>q$, the
stated piecewise condition.  Suppose first that $k=q$, and put
$a=\lfloor s/B\rfloor$.  Remove $a$ copies of $K_m^3$ from each source
color.  The remaining factorization is simple in each color, with common
source degree $r'=r-aA$, target degree $s'=s-aB$, and remaining multiplicity
$\mu=\lambda-ka$.  The identities $qr=\lambda A$ and $ks=\lambda B$ give
$qr'=\mu A$ and $ks'=\mu B$.  Since $mA$ and $nB$ are divisible by three,
the remaining degrees satisfy the same divisibility conditions.  Moreover,
\[
 \frac{r'}A=\frac{s'}B=\frac{\mu}{k}=\frac{s}{B}-a\in[0,1).
\]
If $\mu=0$, then $r'=s'=0$, and assigning $a$ complete copies of $K_n^3$
to every color gives the required embedding.  Otherwise $0<r'\le s'<B$ and
$r'B=s'A$.  The existence part already proved gives an embedding with these
remaining degrees.  Lemma~\ref{CCAII:lem:scalar-simple-counts}, applied with multiplicity
$\mu$, gives a simple embedding for the remaining degrees.
Adding back the $a$ complete copies makes every target factor equimultiple.

Now suppose $k>q$.  The stated piecewise condition implies $s\le B-A<B$:
this is immediate when $n\ge2m$, while for $n<2m$ it follows from
\[
 3m(m-1)-(3m-n)(n+m-3)=(n-2m)(n-3),
\]
because then $(3m-n)s\le3(n-m)\binom m2$ implies $s\le B-A$.
Thus the condition on old multiplicities says exactly that the source factors
are simple.  Lemma~\ref{CCAII:lem:scalar-simple-counts} gives a simple target
factorization exactly when
\[
 s\le B-A,\qquad (3m-n)s\le3(n-m)\binom m2.
\]
The same identity shows that the second inequality implies the first when
$n<2m$.  If $n\ge2m$ and $3m-n>0$, it gives
$(3m-n)(B-A)\le3C_1$, so the first inequality implies the second.  If
$3m-n\le0$, the second inequality is automatic.  Thus these two inequalities
are equivalent to the stated piecewise condition in the theorem.  Since
$s<B$, the resulting simple target factorization is equimultiple.
Finally, Lemma~\ref{CCAII:lem:component-optimal} minimizes the number of
components in every color while preserving equimultiplicity.
\end{proof}

\section{Embeddings when \texorpdfstring{$n\ge(h-1)m$}{n >= (h-1)m}}\label{CCAII:sec:large-order}

Throughout this section, assume $h\ge4$ and $n\ge(h-1)m$.
Then $n-m\ge(h-2)m\ge h$, so $\tau=h$.  Recall that
$A=\binom{m-1}{h-1}$ and $B=\binom{n-1}{h-1}$, and put $D=B-A$.
By Theorem~\ref{CCAII:thm:new-edge-count-criterion}, it suffices to solve
\eqref{CCAII:eq:count-system}.

\begin{theorem}\label{CCAII:thm:large-order-optimal}
An $\vb r$-factorization of $\lambda K_m^h$ embeds in an
$\vb s$-factorization of $\lambda K_n^h$ satisfying
$c(F_i)=\kappa_i$ for $i\in[k]$ if and only if admissibility holds.
Whenever admissibility holds, every $X\in\mathscr X$ is realized by
such an embedding with $\operatorname{spr}(F_i)=\Delta_i(X)$ for $i\in[k]$.
\end{theorem}

\subsection{Choosing the added-edge counts}

A mixed edge containing $u$ old vertices has type $h-u$.
For $u\in[h-1]$, set
\[
 M_u=C_{h-u}=\binom{m}{u}\binom{n-m}{h-u},\qquad
 d_u=\binom{m-1}{u-1}\binom{n-m}{h-u}.
\]
Here $M_u$ is the number of mixed edges containing $u$ old vertices,
and $d_u$ is the number of such edges through a fixed old vertex.
Let $N=\sum\nolimits_{u=2}^{h-1}M_u$, the number of mixed edges containing at
least two old vertices. We have
\begin{equation}\label{CCAII:eq:mixed-incidences}
 uM_u=md_u,\qquad \sum\nolimits_{u\in[h-1]}d_u=D.
\end{equation}

For admissible degrees, put $\delta_i=s_i-r_i$ and recall that
$e_i=(ns_i-mr_i)/h$. Then $\delta_i,e_i\in\mathbb N\cup\{0\}$ and
$\sum\nolimits_{i\in[k]}\delta_i=\lambda D$.
For $x_{i,h-u}\in\mathbb N\cup\{0\}$, put
$W_i=\sum\nolimits_{u\in[h-1]}x_{i,h-u}$.  It is enough to choose these numbers so that
\begin{equation}\label{CCAII:eq:mixed-targets}
\begin{cases}
 \displaystyle\sum\nolimits_{i\in[k]}x_{i,h-u}=\lambda M_u
   &\text{for }u\in[h-1],\\[3pt]
 \displaystyle\sum\nolimits_{u\in[h-1]}u x_{i,h-u}=m\delta_i
   &\text{for }i\in[k],\\[3pt]
 W_i\le e_i
   &\text{for }i\in[k].
\end{cases}
\end{equation}
The entries for type $h$ are then determined by
\begin{equation}\label{CCAII:eq:complete-new-column}
 x_{ih}=e_i-W_i.
\end{equation}
These entries are nonnegative, and
\[
 \sum\nolimits_{i\in[k]} x_{ih}
 =\sum\nolimits_{i\in[k]} e_i-\sum\nolimits_{i\in[k]} W_i
 =\lambda\sum\nolimits_{t\in[h]}C_t
   -\lambda\sum\nolimits_{t\in[h-1]}C_t
 =\lambda C_h.
\]
Also, $\sum\nolimits_{t\in[h]}t x_{it}=he_i-m\delta_i=(n-m)s_i$.
Thus $X$ satisfies \eqref{CCAII:eq:count-system}.

\begin{lemma}\label{CCAII:lem:incidence-reserves}
For $2\le u\le h-1$,
\begin{equation}\label{CCAII:eq:incidence-reserve}
 \sum\nolimits_{v\in[u-1]}vM_v>(u-1)D.
\end{equation}
If, in addition, $n<hm$, then
\begin{equation}\label{CCAII:eq:many-mixed-edges}
 N\ge mD/h.
\end{equation}
\end{lemma}

\begin{proof}
Counting a second old vertex in an edge through a fixed old vertex gives
\begin{equation}\label{CCAII:eq:old-moment}
 \sum\nolimits_{u\in[h-1]}(u-1)d_u
 =(m-1)\binom{n-2}{h-2}-(h-1)A
 =\frac{(h-1)(m-1)}{n-1}B-(h-1)A<D.
\end{equation}
The last inequality uses $(h-1)(m-1)<n-1$. Also $n-m\ge(h-2)m$, so
\begin{equation}\label{CCAII:eq:d-ratio}
 \frac{d_2}{d_1}
 =\frac{(h-1)(m-1)}{n-m-h+2}
 \le\frac{h-1}{h-2}\le\frac32.
\end{equation}
We prove \eqref{CCAII:eq:incidence-reserve} in two cases.
\begin{enumerate}[label=\textup{(\roman*)},leftmargin=1.7em,labelsep=.35em,itemsep=3pt,topsep=2pt]
\item If $u=2$, subtracting $D$ from the sum in
\eqref{CCAII:eq:old-moment} gives
$\sum\nolimits_{v=3}^{h-1}(v-2)d_v<d_1$. Hence
\[
 D=d_1+d_2+\sum\nolimits_{v=3}^{h-1}d_v
 <2d_1+d_2\le\tfrac72d_1<md_1=M_1.
\]

\item If $u\ge3$, \eqref{CCAII:eq:old-moment} gives
$(u-1)\sum\nolimits_{v=u}^{h-1}d_v<D$, and therefore
\[
 \sum\nolimits_{v\in[u-1]}vM_v
 =m\sum\nolimits_{v\in[u-1]}d_v
 >\frac{m(u-2)}{u-1}D\ge(u-1)D.
\]
Indeed, $m\ge h\ge u+1$ and
$(u+1)(u-2)-(u-1)^2=u-3\ge0$.
\end{enumerate}

If $n<hm$, then $n-m\le(h-1)m-1$, so \eqref{CCAII:eq:d-ratio} also gives
$d_1\le d_2$, or $M_1\le2M_2$. Consequently,
\[
 hN-mD=-M_1+\sum\nolimits_{u=2}^{h-1}(h-u)M_u
 \ge -M_1+(h-2)M_2\ge0.
\]
\end{proof}

\begin{lemma}\label{CCAII:lem:direct-mixed-allocation}
Let $\delta_i,p_i\in\mathbb N\cup\{0\}$ satisfy
\[
 \sum\nolimits_{i\in[k]}\delta_i=\lambda D,\qquad
 0\le p_i\le\lceil m\delta_i/h\rceil.
\]
There are $x_{i,h-u}\in\mathbb N\cup\{0\}$ satisfying
\[
 \sum\nolimits_{i\in[k]}x_{i,h-u}=\lambda M_u
 \quad\text{for }u\in[h-1],
 \qquad
 \sum\nolimits_{u\in[h-1]}u x_{i,h-u}=m\delta_i
 \quad\text{for }i\in[k].
\]
If $\lambda N\ge\sum\nolimits_{i\in[k]}p_i$, they may be chosen so that, with
$W_i=\sum\nolimits_{u\in[h-1]}x_{i,h-u}$,
\[
 W_i\le m\delta_i-p_i\qquad\text{for }i\in[k].
\]
\end{lemma}

\begin{proof}
Put $b_i=m\delta_i$. Assign the mixed edges containing at least two old
vertices in the order $u=h-1,\ldots,2$. Let $a_i$ be the total old-vertex
incidence already assigned to color $i$, and let $z_i$ be the sum of $u-1$
over the mixed edges already assigned to color $i$; equivalently, $z_i$ is
$a_i$ minus the number of such edges. Initially both are zero. Assign the next
edge to the first color with $z_i<p_i$, if one exists; otherwise choose any
color with $a_i+u\le b_i$. Each assignment adds $u$ to $a_i$ and $u-1$ to
$z_i$.

To check that each assignment is legal, consider two cases.
\begin{enumerate}[label=\textup{(\roman*)},leftmargin=1.7em,labelsep=.35em,itemsep=3pt,topsep=2pt]
\item If $z_i<p_i$, every previously assigned edge contains at least two old
vertices, so $a_i\le2z_i$. Hence
\[
 a_i+u\le2(p_i-1)+(h-1)
 =2p_i+h-3\le2\lceil b_i/h\rceil+h-3\le b_i.
\]
For the last inequality, $p_i>0$ implies $b_i\ge h$.  If $p_i=1$, the
claim is immediate.  If $p_i\ge2$, then $p_i\le\lceil b_i/h\rceil$ gives
$b_i\ge h(p_i-1)+1\ge2p_i+h-3$.

\item Suppose $z_i\ge p_i$ for all $i$, but no color satisfies
$a_i+u\le b_i$. Each color with $\delta_i>0$ would have at most
$u-1$ incidences left to receive. Since each positive $\delta_i$ is at least
one and $\sum\nolimits_{i\in[k]}\delta_i=\lambda D$, there are at most $\lambda D$
colors with $\delta_i>0$. Hence the total remaining incidence would be at
most $(u-1)\lambda D$. However, all mixed edges containing fewer than $u$
old vertices are still unassigned and contribute more than this by
\eqref{CCAII:eq:incidence-reserve}, a contradiction.
\end{enumerate}

After assigning all mixed edges containing at least two old vertices, set
$x_{i,h-1}=b_i-a_i$. These entries sum to $\lambda M_1$ by
\eqref{CCAII:eq:mixed-incidences}. The number of mixed edges in color $i$ is
$W_i=b_i-z_i$.  While some $z_i<p_i$, each assignment decreases
$\sum\nolimits_{i\in[k]}\max\{0,p_i-z_i\}$ by at least one.
Thus $\lambda N\ge\sum\nolimits_{i\in[k]}p_i$ gives $z_i\ge p_i$ for every
color and hence $W_i\le b_i-p_i$.
\end{proof}

\begin{proof}[Proof of Theorem~\ref{CCAII:thm:large-order-optimal}]
Necessity follows by counting target incidences. For sufficiency, use
$\delta_i,e_i$ as above and put $p_i=\max\{0,m\delta_i-e_i\}$.
Since
\[
 he_i=n\delta_i+(n-m)r_i\ge n\delta_i,
\]
we have $p_i=0$ when $n\ge hm$. Suppose instead that
$(h-1)m\le n<hm$. Then
\[
 m\delta_i-e_i
 =\frac{(hm-n)\delta_i-(n-m)r_i}{h}
 \le\frac{hm-n}{h}\delta_i
 \le\frac{m\delta_i}{h}.
\]
Hence
\[
 \sum\nolimits_{i\in[k]} p_i
 \le\frac m h\sum\nolimits_{i\in[k]}\delta_i
 =\frac{m\lambda D}{h}
 \le\lambda N,
\]
where the last inequality is Lemma~\ref{CCAII:lem:incidence-reserves}.
Lemma~\ref{CCAII:lem:direct-mixed-allocation} therefore gives
\eqref{CCAII:eq:mixed-targets}, since $W_i\le m\delta_i-p_i\le e_i$.
Define the last column by \eqref{CCAII:eq:complete-new-column}.  Thus
$\mathscr X\ne\varnothing$, and
Theorem~\ref{CCAII:thm:new-edge-count-criterion} now gives the claimed
embedding and, for every $X\in\mathscr X$, an embedding with the stated
spread and component counts.  Together with Corollary~\ref{CCAII:cor:h23-large-order},
this also proves Theorem~\ref{CCAII:thm:main-existence}.
\end{proof}

Under admissibility, all target factors can be chosen connected if and only if
$s_i>r_i$ for $i\in[q]$ and $s_i\ge2$ for $i\in[k]\setminus[q]$; this is
exactly the condition $\kappa_i=1$ for every $i\in[k]$.

\begin{corollary}\label{CCAII:cor:large-order-scalar}
An $r$-factorization of $\lambda K_m^h$ embeds in an $s$-factorization of
$\lambda K_n^h$ if and only if $h\mid ns$ and $s\mid\lambda B$.
Under these conditions, all target factors can be chosen connected if and
only if $s>r$.
\end{corollary}

\subsection{Bounds for the added-edge counts}

The main embedding theorem has already been proved.  For simple and
equimultiple embeddings, we additionally need $x_{it}\le C_t$ in the
specified colors.  We obtain these stronger bounds by assigning factors
to colors. Delete the new vertices from each mixed edge; the remaining
edges form the nonuniform hypergraph
\[
 \mathcal P=\bigcup\nolimits_{u=1}^{h-1}\binom{n-m}{h-u}K_m^u,
\]
which has degree $D$ at every old vertex.

\begin{corollary}\label{CCAII:cor:mixed-row-factorization}
Let $m\ge h\ge4$ and $n\ge(h-1)m$, with $(h,m,n)\ne(4,5,15)$.
Then $\mathcal P$ has a $1$-factorization in which all factors contain at
most $L=\lfloor n/h\rfloor$ edges.
\end{corollary}

The proof is given in Appendix~\ref{CCAII:app:rounding}.

\begin{lemma}\label{CCAII:lem:bounded-type-counts}
Suppose admissibility holds, and let $I\subseteq[k]$ satisfy $r_i\le A$ and
$s_i-r_i\le D$ for $i\in I$.  Then \eqref{CCAII:eq:count-system} has a solution with
$x_{it}\le C_t$ for $i\in I$ and $t\in[h]$.
\end{lemma}

\begin{proof}
Use $\delta_i,e_i$ as above, and put
$T=\sum\nolimits_{u\in[h-1]}M_u=\sum\nolimits_{t\in[h-1]}C_t$.
We first factor $\mathcal P$ in a form suited to the required bounds.
\begin{enumerate}[label=\textup{(\roman*)},leftmargin=1.7em,labelsep=.35em,itemsep=3pt,topsep=2pt]
\item If $(h,m,n)\ne(4,5,15)$, use the $D$ $1$-factors of
Corollary~\ref{CCAII:cor:mixed-row-factorization} and put $\nu=1$.

\item If $(h,m,n)=(4,5,15)$, apply
Lemma~\ref{CCAII:lem:direct-mixed-allocation} with $\lambda=1$, $90$ colors,
$\delta_j=4$, and $p_j=5$. Here $N=550\ge90\cdot5$, so the counts have
total incidence $20$ with the old vertices and at most $15$ edges per
color. Theorem~\ref{CCAII:thm:balanced-R} gives $90$ $4$-factors
of $\mathcal P$; put $\nu=4$.
Since $4\mid5r_i$ and $4\mid15s_i$, we have $4\mid r_i,s_i$ and hence
$\nu\mid\delta_i$.
\end{enumerate}

In either case $\mathcal P$ has $Q=D/\nu$ factors of degree $\nu$, each
with at most $n\nu/h$ edges.  List these factors in a fixed order, repeat
this list $\lambda$ times, and divide it into consecutive blocks of lengths
$\delta_i/\nu$, giving block $i$ to color $i$.  These lengths are integers
and sum to $\lambda Q$.  For $i\in I$, the block has length at most $Q$, so
no factor occurs twice in it.
For $t<h$, let $x_{it}$ count the assigned edges containing $h-t$
old vertices in color $i$. These entries sum to $\lambda C_t$ in each
column, have total incidence $m\delta_i$ with the old vertices in color
$i$, and satisfy
$W_i\le n\delta_i/h\le e_i$.
Thus \eqref{CCAII:eq:complete-new-column} completes a solution of
\eqref{CCAII:eq:count-system}.

For $i\in I$, color $i$ receives at most one copy of each factor,
so $x_{it}\le C_t$ for $t<h$. The remaining factors have total degree
$D-\delta_i$ and at most
$n(D-\delta_i)/h$ edges, so
\[
 W_i\ge T-\frac{n(D-\delta_i)}h.
\]
Since $h(T+C_h)=nB-mA$ and
$he_i=n\delta_i+(n-m)r_i$, we obtain
\[
 x_{ih}=e_i-W_i
 \le C_h-\frac{(n-m)(A-r_i)}h\le C_h.
\]
This gives the required bound in every type and every specified color.
\end{proof}

\subsection{Equimultiple and simple embeddings}

\begin{corollary}\label{CCAII:cor:large-order-equimultiple}
Let $I\subseteq[k]$.  An embedding in which $F_i$ is equimultiple for
$i\in I$ exists if and only if admissibility holds and
\[
 \mult_{G_i}(e)\approx s_i/B
 \quad\text{for }i\in I,\ e\in E(K_m^h),\qquad
 (s_i-r_i)/D\approx s_i/B\quad\text{for }i\in I.
\]
An embedding in which $F_i$ is simple for $i\in I$ exists if and only if
admissibility holds and
\[
 G_i\text{ is simple},\qquad s_i-r_i\le D
 \quad\text{for }i\in I.
\]
In either case, the embedding can be chosen with
$c(F_i)=\kappa_i$ for $i\in[k]$.
\end{corollary}

\begin{proof}
For equimultiplicity, necessity follows from the prescribed multiplicities
on old edges and the $D$ underlying added edges through each old vertex.
For sufficiency, put $a_i=\lfloor s_i/B\rfloor$ for $i\in I$ and
$a_i=0$ otherwise.  Remove $a_i$ copies of $K_m^h$ from $G_i$, and set
\[
 r_i'=r_i-a_iA,\qquad s_i'=s_i-a_iB,\qquad
 \mu=\lambda-\sum\nolimits_{i\in[k]}a_i.
\]
The remaining old edges form a coloring of $\mu K_m^h$.
For $i\in I$, the remaining factor is simple, and the two displayed
approximation conditions give
\[
 0\le r_i'\le A,\qquad 0\le s_i'<B,\qquad 0\le s_i'-r_i'\le D.
\]
Since $a_iB\le s_i$ for every $i$, we have
$\sum\nolimits_{i\in[k]}a_i\le\lambda$, so $\mu\ge0$.  Moreover,
\[
 nB=h\binom nh,\qquad mA=h\binom mh,
\]
so removing complete copies preserves the required divisibility conditions.
Thus the remaining degrees are admissible and satisfy $r_i'\le s_i'$.
We distinguish two cases.
\begin{enumerate}[label=\textup{(\roman*)},leftmargin=1.7em,labelsep=.35em,itemsep=3pt,topsep=2pt]
\item If $\mu=0$, then
$\sum\nolimits_{i\in[k]}(s_i-a_iB)=\mu B=0$ and every $s_i-a_iB\ge0$, so
$s_i=a_iB$ for every $i$.  Since $s_i>0$, this forces $a_i>0$ for every
$i$, so every color lies in $I$; the second approximation condition then
gives $s_i-r_i=a_iD>0$.  Thus $\kappa_i=1$ for every $i$, and assigning
$a_i$ copies of $K_n^h$ to color $i$ gives the required embedding.

\item Suppose $\mu>0$.  Temporarily omit any color with $s_i'=0$.
Among the colors that remain, regard those with $r_i'>0$ as source colors
and those with $r_i'=0<s_i'$ as additional target colors, regardless of
their status before the complete copies were removed.  Let $I'$ consist of
the remaining colors from $I$.  Relabel these colors and $I'$ if necessary,
and apply Lemma~\ref{CCAII:lem:bounded-type-counts} with $\mu$ in place of
$\lambda$ and $I'$ in place of $I$.  The resulting counts are at most
$C_t$ for $i\in I'$.  In these colors the remaining old factors are simple,
and every type has average multiplicity at most one.  Hence the construction in
Theorem~\ref{CCAII:thm:new-edge-count-criterion} uses only $0$--$1$
multiplicities and gives simple, hence equimultiple, target factors while
minimizing the number of components in every remaining color.

Restore the original color labels, reinstate any omitted colors, and add back
the $a_i$ copies of $K_n^h$.  This preserves equimultiplicity.  If $a_i>0$,
the condition $(s_i-r_i)/D\approx s_i/B$ gives
$s_i-r_i\ge a_iD>0$.  Thus $\kappa_i=1$, and the added copies make color
$i$ connected.  Colors with $a_i=0$ retain their minimum component counts.
Hence the resulting embedding satisfies $c(F_i)=\kappa_i$ for all colors.
\end{enumerate}

For simplicity, necessity follows by restriction to the old vertices and by
counting the $D$ possible added edges through an old vertex.  Conversely,
assume admissibility, simplicity of $G_i$, and $s_i-r_i\le D$ for $i\in I$.
Then $r_i\le A$ for $i\in I$, so
Lemma~\ref{CCAII:lem:bounded-type-counts} gives $X\in\mathscr X$ with
$x_{it}\le C_t$ in these colors.  The construction in
Theorem~\ref{CCAII:thm:new-edge-count-criterion} balances each type, so every
added-edge multiplicity is at most one; together with simplicity of $G_i$,
this gives simple target factors for $i\in I$, while
$c(F_i)=\kappa_i$ for $i\in[k]$.
\end{proof}

Let $I\subseteq[k]$, and put $a_i=\lfloor s_i/B\rfloor$ and
$b_i=\lceil s_i/B\rceil$.  For $i\in I$, the degree condition for equimultiplicity is equivalently
$a_iD\le s_i-r_i\le b_iD$.  For an additional target color
$i\in I\setminus[q]$, the two equimultiplicity conditions reduce to
$s_i\le D=B-A$, and the target factor is simple.

Taking $I=[k]$, all target factors can be chosen connected and simple if and
only if admissibility holds and
\[
 \begin{aligned}
 G_i&\text{ is simple},\qquad r_i<s_i\le r_i+D
 &&\text{for }i\in[q],\\
 2&\le s_i\le D
 &&\text{for }i\in[k]\setminus[q].
 \end{aligned}
\]
This combines the simplicity criterion with the connectedness condition
following Theorem~\ref{CCAII:thm:large-order-optimal}.

\section{Target degrees and iterated embeddings}\label{CCAII:sec:further-consequences}

\subsection{Target degrees}

We determine the possible numbers of nonempty target colors under an upper
degree bound, and the minimum possible maximum target degree.  The same
calculation applies when all target factors are required to be connected.

\begin{theorem}\label{CCAII:thm:embedding-color-spectrum}
Suppose either $h\in\{2,3\}$ and $n\ge2m$, or $h\ge4$ and
$n\ge(h-1)m$, and let the prescribed factorization $\cG$ have $q$ nonempty
source colors of degrees $r_1,\dots,r_q$.  Put
$B=\binom{n-1}{h-1}$, $S=\lambda B$, and $g=h/\gcd(h,n)$.
Use one of the following choices of lower bounds:
\[
 \begin{array}{c|cc}
  \text{target factors} & \ell_i\ (i\in[q]) & a\\[2pt]\hline
  \text{no connectivity requirement}
   & g\lceil r_i/g\rceil & g\\[2pt]
  \text{all connected}
   & g\lceil(r_i+1)/g\rceil & \max\{2,g\}.
 \end{array}
\]
Here $\ell_i$ is the lower bound for source color $i$, and $a$
is the lower bound for an additional target color.
Given $b\in\mathbb N$, put $R=g\lfloor b/g\rfloor$.
For either choice, the following hold.

\begin{enumerate}
\item[\textup{(i)}] For $k\ge q$, an embedding with $k$ nonempty
target colors, each of degree at most $b$, exists if and only if
\[
 \ell_i\le R\quad\text{for }i\in[q],
 \qquad
 \sum\nolimits_{i\in[q]}\ell_i+(k-q)a\le S\le kR.
\]

\item[\textup{(ii)}] If $\ell_i\le R$ for $i\in[q]$, the possible numbers
of nonempty target colors under the bound $b$ are exactly the integers
\[
 \max\left\{q,\left\lceil\frac SR\right\rceil\right\}
 \le k\le
 q+\left\lfloor\frac{S-\sum\nolimits_{i\in[q]}\ell_i}{a}\right\rfloor.
\]
If some $\ell_i>R$, or the interval is empty, there is no such embedding.

\item[\textup{(iii)}] Independently of $b$, for fixed $k\ge q$, an embedding
exists if and only if
$\sum\nolimits_{i\in[q]}\ell_i+(k-q)a\le S$.  Whenever this holds,
\[
 \min\nolimits_{\vb s}\max\nolimits_{i\in[k]}s_i
 =
 \max\left\{\max\nolimits_{i\in[q]}\ell_i,\
       g\left\lceil\frac{S}{kg}\right\rceil\right\},
\]
where the minimum is over target degree vectors permitting an embedding
with the chosen connectivity requirement.
\end{enumerate}
\end{theorem}

\begin{proof}
The condition $h\mid ns_i$ is equivalent to $g\mid s_i$.
Since $nB=h\binom nh$, we also have $g\mid B$, and hence $g\mid S$.
Thus admissible positive degrees are positive multiples of $g$, and
$R$ is the largest nonnegative multiple of $g$ at most $b$.
The ordinary lower bounds follow from $s_i\ge r_i$ for source colors
and $s_i\ge1$ for additional target colors.
Whenever an embedding exists, Lemma~\ref{CCAII:lem:component-optimal}
allows all target factors to be connected exactly when $s_i>r_i$ for
$i\in[q]$ and $s_i\ge2$ for $i\in[k]\setminus[q]$, giving the second
choice of lower bounds.  In both cases, $\ell_i$ and $a$ are multiples of
$g$ and $\ell_i\ge a$ for $i\in[q]$, since $r_i\ge1$.

For \textup{(i)}, the lower bounds and the degree sum give necessity.
Conversely, start the source colors at $\ell_i$ and the additional target
colors at $a$.  Their sum is at most $S$, and the total available increase
before any degree reaches $R$ is at least the remaining amount.
All amounts are multiples of $g$, so distribute the remainder in steps of
$g$.  The resulting positive degrees sum to $S$ and satisfy admissibility.
Theorem~\ref{CCAII:thm:main-existence} gives an embedding; in the connected
case, Lemma~\ref{CCAII:lem:component-optimal} gives connected target factors.

Solving the sum inequalities in \textup{(i)} for $k$ gives \textup{(ii)}.
The condition $\ell_i\le R$ implies $R>0$, since $q\ge1$ and $\ell_i\ge g$.
If $b<g$, this condition fails, so there is no embedding.

For \textup{(iii)}, every common upper bound is at least each $\ell_i$
and at least the least multiple of $g$ whose $k$ copies sum to at least $S$.
Let $M$ be the displayed maximum.  Then $M$ is a multiple of $g$,
$M\ge\ell_i\ge a$, and $S\le kM$.  If the sum of the lower bounds is at most
$S$, part~\textup{(i)} with $b=M$ gives an embedding attaining this bound.
\end{proof}

For $h\ge4$ and $n\ge(h-1)m$, the criteria for equimultiple target factors
are given in Corollary~\ref{CCAII:cor:large-order-equimultiple}.
These also depend on the prescribed old multiplicities, rather than only
on the source and target degrees.

\subsection{Iterated embeddings}\label{CCAII:sec:iterated-embeddings}

We now compare embeddings through intermediate orders with direct
embeddings.  The component comparison applies to any existing chain of
increasing orders.

\begin{proposition}\label{CCAII:prop:iterated-components}
Let $t\in\mathbb N$, and let $n_0,\dots,n_t\in\mathbb N$ with
$h\le n_0<n_1<\cdots<n_t$.  Let
$\cG_0\subseteq\cG_1\subseteq\cdots\subseteq\cG_t$ be a chain in which
$\cG_{j-1}$ is embedded in $\cG_j$, a factorization of
$\lambda K_{n_j}^h$, for $j\in[t]$.  Fix a color $i$ that is nonempty in
$\cG_0$, and write $G_{ji}$ for the color-$i$ factor of $\cG_j$,
$d_j$ for its degree, and $c_j=c(G_{ji})$.
Suppose that, at each stage, among all embeddings of $\cG_{j-1}$ in a
factorization of $\lambda K_{n_j}^h$ with target color-$i$ degree $d_j$,
$\cG_j$ has the minimum possible number $c_j$ of components in color $i$.

\textup{(i)} For $j\in[t]$,
\[
 c_j=
 \begin{cases}
  1, & d_j>d_{j-1},\\
  c_{j-1}+1, & d_j=d_{j-1}\ge2,\\
  c_{j-1}+(n_j-n_{j-1})/h, & d_j=d_{j-1}=1.
 \end{cases}
\]

\textup{(ii)} If $\ell$ is the last index for which $d_\ell>d_{\ell-1}$, then
\[
 c_t=1+t-\ell.
\]
If the degree never increases, then
\[
 c_t=
 \begin{cases}
  c_0+t, & d_0\ge2,\\
  c_0+(n_t-n_0)/h, & d_0=1.
 \end{cases}
\]

\textup{(iii)} Compared in color $i$ with a direct embedding of $\cG_0$ having
the same final target degree and minimum component count in that color, the
excess number of components is $t-\ell$ if the degree increases at least once,
$t-1$ if the degree is constant and at least two, and $0$ if the degree is
constantly one.
\end{proposition}

\begin{proof}
Restriction preserves the old edges, so $d_0\le\cdots\le d_t$.
Part~\textup{(i)} follows by applying
Lemma~\ref{CCAII:lem:component-optimal} at stage $j$.  If
a degree increases, every old vertex receives positive new color-$i$
incidence; after the new vertices are amalgamated, every previous component
meets the amalgamated vertex $\alpha$.  Lemma~\ref{CCAII:lem:component-optimal}
therefore gives $c_j=1$.  If the degree is unchanged and at least two, one component is
added; if it is unchanged and equal to one, the increment is
$(n_j-n_{j-1})/h$.
Iterating the recurrence gives part~\textup{(ii)}.  After the last degree
increase, the degree is at least two, so every remaining stage adds one
component.  If the degree never increases and equals one, the increments
telescope.  In this case $h\mid n_j$ for every $j$, since $G_{ji}$ is a
$1$-factor.

For part~\textup{(iii)}, composition of the stagewise embeddings gives a
direct embedding of $\cG_0$ at order $n_t$ with the same final target degree.
Applying Lemma~\ref{CCAII:lem:component-optimal} once to such a direct
embedding gives minimum component count $1$ when $d_t>d_0$, $c_0+1$ when
$d_t=d_0\ge2$, and $c_0+(n_t-n_0)/h$ when $d_t=d_0=1$.
Subtracting the values in part~\textup{(ii)} gives the stated excesses.
\end{proof}

For multiplicities, suppose $M\in\mathbb N\cup\{0\}$ and
$b_1,\dots,b_t\in\mathbb N$.  At stage $j$, split every current part
fairly into $b_j$ parts.  Each final part $x$ then satisfies
$x\approx M/(b_1\cdots b_t)$.
Indeed, the nested floor--ceiling identities in
Section~\ref{CCAII:sec:preliminaries} show inductively that the bounds
after stage $j$ are $\lfloor M/(b_1\cdots b_j)\rfloor$ and
$\lceil M/(b_1\cdots b_j)\rceil$.
The same number of parts must be used for every current part at each stage:
splitting $8$ as $4,4$ and then only one of these parts as $2,2$
gives $2,2,4$, which is not a fair three-way split.

This rounding observation concerns repeated subdivision of a fixed
multiplicity.  In an embedding through intermediate orders, however, the
earlier edge multiplicities become prescribed data.  Balancing each class
of added edges at every stage therefore need not match the spread of a
direct embedding.
For example, take $h=3$, $\lambda=3$, and source colors $K_4^3$ and
$2K_4^3$, and consider
\[
 4\longrightarrow7\longrightarrow10,
 \qquad
 (3,6)\longrightarrow(3,42)\longrightarrow(3,105).
\]
At the first stage, the first color has unchanged degree $3$, so
no mixed edge receives that color.  The unique edge on the three new
vertices therefore has multiplicity $3$ in that color.
The second color has multiplicity $3$ on every mixed edge and
multiplicity $0$ on the all-new edge.  Thus both colors have spread $3$.
A chain with the displayed degrees exists: at each stage, retain the
first color and add three copies of the triple on the three new
vertices; give all remaining edges to the second color.
At the second stage, the old multiplicities $0$ and $3$ are fixed, so both
final spreads remain $3$.

For a direct embedding from order $4$ to order $10$, label the six new
vertices $v_1,\dots,v_6$ and put one copy of
$\{v_j,v_{j+1},v_{j+2}\}$ in the first color for $j\in[6]$, with indices
read modulo $6$.  Together with the prescribed $K_4^3$, these give a
simple $3$-factor: each new vertex lies in exactly three of the six new
edges.  Its complement in $3K_{10}^3$ is a $105$-factor and restricts to
$2K_4^3$ on the old vertices.
The first color has multiplicities $0$ and $1$, and the second has
multiplicities $2$ and $3$.  Hence both direct spreads are $1$.

Finally, let $h\ge4$, $t\in\mathbb N$, and $h\le n_0<\cdots<n_t$.
Start with a factorization of $\lambda K_{n_0}^h$ and prescribe positive
admissible degrees at each later order.  Retain all existing color labels,
require their degrees not to decrease, and allow additional target colors.
Repeated application of Theorem~\ref{CCAII:thm:large-order-optimal}
gives a chain whenever
\[
 n_j\ge(h-1)n_{j-1}\qquad\text{for }j\in[t].
\]
Thus the stagewise theorem requires $n_t\ge(h-1)^t n_0$, whereas the
direct theorem already applies when $n_t\ge(h-1)n_0$ for the same final
target degrees.

\appendix
\section{Rounding in the factorization of mixed edges}\label{CCAII:app:rounding}

\begin{proof}[Proof of Corollary~\ref{CCAII:cor:mixed-row-factorization}]
Assume $m\ge h\ge4$ and $n\ge(h-1)m$.  Recall
\[
 \mathcal P=\bigcup\nolimits_{u\in[h-1]}
 \binom{n-m}{h-u}K_m^u,\qquad
 M_u=\binom mu\binom{n-m}{h-u},\qquad
 d_u=\binom{m-1}{u-1}\binom{n-m}{h-u}.
\]
Here $M_u$ counts mixed edges with $u$ old vertices,
$D=\sum\nolimits_{u\in[h-1]}d_u=B-A$ is the degree of $\mathcal P$, and
$N=\sum\nolimits_{u=2}^{h-1}M_u$.  Put
$L=\lfloor n/h\rfloor$ and $p=\max\{0,m-L\}$.

Apply the greedy assignment of Lemma~\ref{CCAII:lem:direct-mixed-allocation}
with $D$ colors and $\delta_i=1$.  An assigned edge with $u$ old vertices
uses one edge to supply $u$ old-vertex incidences, so it reduces the final
number of edges in its color by $u-1$.  Let $z_i$ be the accumulated
reduction in color $i$.  The final number of edges in that color is
$m-z_i$, so it is enough to prove $z_i\ge p$ for every $i\in[D]$.
The greedy rule sends an edge to the first color with $z_i<p$ whenever such
a color exists, and Lemma~\ref{CCAII:lem:direct-mixed-allocation} shows that
every assignment is legal.  Hence $N\ge pD$ is sufficient, because each of
the $N$ edges with at least two old vertices reduces the remaining
requirement by at least one.

If $p=0$, there is nothing to prove.  If $p\le m/h$, then
Lemma~\ref{CCAII:lem:incidence-reserves} gives $N\ge mD/h\ge pD$; here
$p>0$ implies $n<hm$.  Since $n\ge(h-1)m$, the only unsettled case is
$p=\lceil m/h\rceil$ with $h\nmid m$.  If $r_0$ is the nonzero remainder
of $m$ modulo $h$, this occurs exactly when
\begin{equation}\label{CCAII:eq:rounding-strip}
 (h-1)m\le n\le(h-1)m+r_0-1.
\end{equation}
We treat this short interval in four steps.

\begin{enumerate}[label=\textup{(\roman*)},leftmargin=1.7em,labelsep=.35em,itemsep=3pt,topsep=2pt]
\item Suppose $h\ge5$.  The proof of
Lemma~\ref{CCAII:lem:incidence-reserves} gives $D<2d_1+d_2$.
In \eqref{CCAII:eq:rounding-strip} we also have $d_1\le d_2$, so
$D<3d_2$ and $M_2=md_2/2>mD/6$.  Since $M_1\le2M_2$,
\[
 hN-mD=-M_1+\sum\nolimits_{u=2}^{h-1}(h-u)M_u\ge(h-4)M_2.
\]
Therefore $N/D>m(h+2)/(6h)$, which is at least
$p\le(m+h-1)/h$ whenever $m(h-4)\ge6(h-1)$.  This covers all unsettled
cases with $h\ge9$: since $h\nmid m$, we have $m\ge h+1$, and
$(h+1)(h-4)\ge6(h-1)$ for $h\ge9$.

\item Suppose $h=4$.  Put $t=d_1/d_2$ and $v=d_3/d_2$.
In \eqref{CCAII:eq:rounding-strip}, $n-m\le2m+2$, so
$t\le2m/(3(m-1))$ and $v\ge(m-2)/(2m+1)$.  The expression
$(1/2+v/3)/(1+t+v)$ decreases with $t$ and, after replacing $t$ by this
upper bound, increases with $v$.  Consequently
\[
 \frac ND\ge
 \frac{m(m-1)(8m-1)}{2(13m^2-10m+3)}
 \ge\frac{m+3}{4}\ge p\qquad\text{for }m\ge16.
\]
The middle inequality is equivalent to
$3m^3-47m^2+29m-9\ge0$.  At $m=16+a$ this polynomial is
$3a^3+97a^2+829a+711$, so it is positive for every $a\ge0$.

\item It remains to check the following ranges, always excluding
$h\mid m$, and every $n$ in \eqref{CCAII:eq:rounding-strip}:
\[
\begin{array}{c|c|c|c}
 h&\text{range of }m&\text{number of triples}&
 \text{smallest nonnegative }N-pD\\ \hline
 4&5\le m\le15&18&40\\
 5&6\le m\le23&36&275\\
 6&7\le m\le14&18&1953\\
 7&8\le m\le11&10&2803332\\
 8&9\le m\le10&3&254931570
\end{array}
\]
These $85$ triples are checked by substituting
\[
 N=\sum\nolimits_{u=2}^{h-1}\binom mu\binom{n-m}{h-u},\qquad
 D=\binom{n-1}{h-1}-\binom{m-1}{h-1},\qquad
 p=\left\lceil\frac mh\right\rceil.
\]
The inequality $N\ge pD$ fails only for
\[
 \mathcal E=\{(4,5,15),(4,6,18),(4,6,19),(4,9,27),(4,13,39),(5,6,24)\}.
\]

\item For these six triples, use the fact that an edge with at least three
old vertices can reduce the remaining requirement by more than one.  Put
\[
 R=\sum\nolimits_{i\in[D]}\max\{0,p-z_i\}
\]
just before the edges with two old vertices are processed.  Such edges then
reduce $R$ by one while $R>0$, so $R\le M_2$ is sufficient.  The relevant
values are
\[
\begin{array}{c|r|r|r}
 (h,m,n)&N-pD&R&M_2\\ \hline
 (4,5,15)&-170&520&450\\
 (4,6,18)&-110&860&990\\
 (4,6,19)&-182&1092&1170\\
 (4,9,27)&-612&5364&5508\\
 (4,13,39)&-78&17992&25350\\
 (5,6,24)&-2130&11040&12240
\end{array}
\]
For $h=4$, an edge with three old vertices contributes two units of
reduction.  Each successive color receives $\lceil p/2\rceil$ such edges
before the next color is used.  At all five listed triples,
$M_3$ is divisible by $\lceil p/2\rceil$ and
$M_3/\lceil p/2\rceil<D$, so
$R=p(D-M_3/\lceil p/2\rceil)$.  For $(5,6,24)$, $p=2$ and $D=8850$;
each of the $270$ edges with four old vertices and the $3060$ edges with
three old vertices completes the requirement in one previously unused
color, leaving $R=2(8850-270-3060)=11040$.  Thus $R\le M_2$ for the five
triples in $\mathcal E\setminus\{(4,5,15)\}$.
\end{enumerate}

For every nonexceptional triple, the assignment prescribes $m$ incidences
and at most $m-p\le L$ edges in each of the $D$ colors, with the correct
column totals.  Theorem~\ref{CCAII:thm:balanced-R} gives
a $1$-factorization of $\mathcal P$.  This proves the corollary.
\end{proof}

The excluded triple is not merely a failure of the estimate: it is a
genuine obstruction to the bounded $1$-factorization.  For $(h,m,n)=(4,5,15)$, $\mathcal P$ has $1150$ edges
and degree $360$.  A $1$-factorization has $360$ factors, so if every
factor contained at most $L=3$ edges it would contain at most $1080$ edges
in total.  The $4$-factorization used in
Lemma~\ref{CCAII:lem:bounded-type-counts} avoids this obstruction.

\end{document}